\documentclass[11pt]{article}
\usepackage{amssymb,amsmath,amsfonts,amsthm,epsfig,latexsym,color}
\usepackage{amsfonts,amsmath,amsthm, amssymb, amscd, mathrsfs}
\usepackage{latexsym, euscript, epic, eepic}
\usepackage{graphicx}
\usepackage{verbatim}
\usepackage{fourier}

\usepackage[colorlinks=true,backref=page]{hyperref}

\makeatletter
\newcommand{\subsectionruninhead}{\@startsection{subsection}{2}{0mm}
{-\baselineskip}{-0mm}{\bf\large}}
\newcommand{\subsubsectionruninhead}{\@startsection{subsubsection}{3}{0mm}
{-\baselineskip}{-0mm}{\bf\normalsize}}
\makeatother

\newtheorem*{theorem*}{Theorem}

\newtheorem{theoremalph}{Theorem}

\newtheorem*{proposition*}{Proposition}
\newtheorem*{corollary*}{Corollary}
\newtheorem*{claim*}{Claim}
\newtheorem*{remark*}{Remark}
\newtheorem*{problem*}{Problem}
\newtheorem{theorem}{Theorem}[section]

\newtheorem{proposition}[theorem]{Proposition}
\newtheorem{corollary}[theorem]{Corollary}
\newtheorem{lemma}[theorem]{Lemma}
\newtheorem{notation}[theorem]{Notation}

\newtheorem{claim}[theorem]{Claim}
\theoremstyle{definition}
\newtheorem{definition}[theorem]{Definition}
\newtheorem{remark}[theorem]{Remark}

\numberwithin{equation}{section}

 \def\RR{{\mathbb R}}

\def\cC{\mathcal{C}}    \def\cU{\mathcal{U}}
    \def\cV{\mathcal{V}}
    
\def\cF{\mathcal{F}}    \def\cX{\mathscr{X}}

\newcommand{\supp}{\operatorname{Supp}}

\newcommand{\orb}{\operatorname{Orb}}

\newcommand{\ind}{\operatorname{Ind}}
\newcommand{\sing}{\operatorname{Sing}}
\newcommand{\per}{\operatorname{Per}}
\newcommand{\jac}{\operatorname{Jac}}

\begin{document}

\title{A $C^r$-connecting lemma for Lorenz attractors and its application on the space of ergodic measures}

\author{Yi Shi\footnote{Y. Shi was partially supported by National Key R\&D Program of China (2021YFA1001900) and NSFC (12071007, 11831001, 12090015).}\, ,\,  
	Xueting Tian\footnote{X. Tian was partially supported by NSFC (12071082) and Science and Technology Innovation Action Program of Science and Technology Commission of Shanghai Municipality (STCSM, No. 21JC1400700).}  \,   
	 and  
	Xiaodong Wang\footnote{X. Wang was partially supported by NSFC (12071285, 11701366), National Key R\&D Program of China (2021YFA1001900) and Innovation Program of Shanghai Municipal Education Commission (No. 2021-01-07-00-02-E00087).}}

\maketitle

\begin{abstract}
	For every $r\in\mathbb{N}_{\geq 2}\cup\{\infty\}$, we prove a $C^r$-connecting lemma for Lorenz attractors. To be precise, for a Lorenz attractor of a $3$-dimensional $C^r$ ($r\geq 2$) vector field,  a heteroclinic orbit associated to the singularity and a critical element can be created through arbitrarily small $C^r$-perturbations. As an application, we show that for $C^r$-dense geometric Lorenz attractors, the Dirac measure of the singularity is isolated inside the space of ergodic measures and thus the ergodic measure space is not connected; while for $C^r$-generic geometric Lorenz attractors, the space of ergodic measures is path connected with dense periodic measures. 
	In particular, the generic part proves a conjecture proposed by C. Bonatti~\cite[Conjecture 2]{b-survey} in $C^r$-topology for Lorenz attractors.
\end{abstract}

\section{Introduction}

In 1963, E. Lorenz~\cite{Lorenz} introduced the following differential equations in the study of meteorology:
\begin{align}\label{Equa:Lorenz}
 \left\{ \begin{array}{ll}
\dot x=\alpha(y-x),\\
\dot y=\beta x-y-xz, \\
\dot z=xy-\gamma z,
\end{array} \right.
\end{align}
which has arisen much interest for dynaminists since then.
For an open neighborhood of the chosen parameters $(\alpha,\beta,\gamma)=(10,28,8/3)$, numerical simulations suggest that the system generated by the solution of~(\ref{Equa:Lorenz})  admits a strange attractor called the {\it Lorenz attractor} which attracts almost all points in the phase space. Dynamics of Lorenz attractors behave in a chaotic way in the sense that it is transitive with dense periodic orbits and thus it is sensitively dependent on initial conditions. Although expressed in a simplified model, Lorenz equations (\ref{Equa:Lorenz}) have been proved to be resistant and  present much difficulty for rigorous mathematical analysis. Thanks to the independent pioneering works of J. Guckenheimer~\cite{guck} and V. Afra$\breve{\rm\i}$movi$\check{\rm c}$-V. Bykov-L. Sil'nikov~\cite{abs}, a geometric model of Lorenz attractor was introduced whose dynamics admits the behaviors observed by Lorenz. The {\it geometric Lorenz attractor} (see Definition~\ref{def:Lorenz}) is defined through a three dimensional smooth vector field and is proved to possess a singularity (or an equilibrium) where the vector field vanishes and which is robustly accumulated by regular orbits. On the other hand, the geometric Lorenz attractor is not  structurally stable~\cite{guck}.  Works of J. Guckenheimer and R. Williams~\cite{gw,williams} gave  well descriptions of the structure of geometric Lorenz attractors which led more studies. 
W. Tucker~\cite{tucher1,tucher2} proved that the system generated from Lorenz original equations~\ref{Equa:Lorenz} admits the geometric model, which answers affirmatively~\cite[Problem 14]{Smale-question} raised by S. Smale.
Due to the existence of the singularity, the geometric Lorenz attractor is not  hyperbolic (see Definition~\ref{Def:hyp}). To describe its geometric structure, C. Morales-M. Pacifico-E. Pujals~\cite{mpp} introduced the notion of {\it singular hyperbolicity} (or {\it sectional hyperbolicity}),  see Definition~\ref{Def:singular-hyp}.
One can refer to~\cite{lgw,Me-Mo} for higher dimensions. 
S. Luzzatto-I. Melbourne-F. Paccaut~\cite{LMP} proved the geometric Lorenz attractor is mixing with respect to its physical measure.
More recently, V. Ara\'ujo-I. Melbourne~\cite{am2} proved some more properties including the existence of SRB measures, the central limit theorem and associated invariance principles.

\medskip

In this paper, we prove a $C^r (r\in\mathbb{N}_{\geq 2}\cup\{\infty\})$ connecting lemma for the geometric Lorenz attractor and study its ergodic measure space as an application. Perturbation theory plays a key role in the study of ``most'' dynamical systems, for instance in the exploration of $C^1$-stability conjecture.
C. Pugh's closing lemma~\cite{Pugh} serves as a foundation of $C^1$-perturbation theory. A series of important $C^1$-perturbation theorems were inspired by this milestone work: the $C^1$-closing lemma for conservative and Hamiltonian diffeomorphisms by C. Pugh-C. Robinson~\cite{PR}, the $C^1$-ergodic closing lemma by R. Ma\~n\'e~\cite{mane-ergodic closing}, the $C^1$-closing lemma for nonsingular endomorphisms by L. Wen~\cite{Wen-endo-closing}, the $C^1$-connecting lemma by S. Hayashi~\cite{hayashi}, see also M. Arnaud~\cite{ar}, L. Wen-Z. Xia~\cite{wen-xia-connecting}. The strongest $C^1$-perturbation theorem would be the $C^1$-chain connecting lemma which was proved by C. Bonatti-S. Crovisier~\cite{bc}, see also~\cite{bdv,crovisier,C-habitation} for more applications.
The proofs of these $C^1$-perturbation theorems relies on local perturbation methods introduced by C. Pugh.

Compared to the completed perturbation theory in $C^1$-topology, when $r\geq 2$, the $C^r$-perturbation theory is far from being accomplished.
The $C^r (r\geq 2)$ perturbation problem turns out to be much more delicate and difficult.
C. Gutierrez~\cite{Gu} built an example which showed that C. Pugh's local perturbation method for $C^1$-perturbations does not work for $C^r (r\geq 2)$ case. 
When $r$ is large, the $C^r$-closing lemma is false for the system of Hamiltonian vector field constructed by M. Herman~\cite{herman1,herman2}.
These two counter-examples imply that in general, global perturbations are needed in exploring the $C^r (r\geq 2)$-perturbation problems.
The lack of $C^r$-perturbation techniques makes the study of $C^r$-``most'' dynamics be much difficult. For instance, the  $C^r$-stability conjecture is still widely open, for which one can refer to the survey by E. Pujals~\cite{pujals}.
For lower dimensional dynamical systems, L-S. Young~\cite{Young} proved the $C^r$-closing lemma for interval endomorphisms. 
M. Asaoka-K. Irie~\cite{AI} proved a $C^\infty$-closing lemma for Hamiltonian surface diffeomorphisms. In particular, S. Crovisier-E. Pujals~\cite{CP} proved the $C^\infty$-closing lemma on the support of every ergodic measure for strongly dissipative surface diffeomorphisms, including a large class of H\'enon maps.
For higher dimensional dynamical systems, R. Ma\~n\'e~\cite{mane-h-point} proved a $C^2$-connecting lemma to create homoclinic points under certain assumptions in the measure sense.
Very recently, S. Gan and the first author~\cite{GS-closing} proved the $C^r$-closing lemma for partially hyperbolic diffeomorphisms with 1-dimensional center.

Let $M^d$ (or $M$ for simplicity) be a $d$-dimensional ($d\geq 3$) $C^{\infty}$  smooth closed Riemannian manifold. 
For every $r\in\mathbb{N}\cup\{\infty\}$, we denote by $\mathscr{X}^r(M)$ the set of $C^r$-smooth vector fields on $M$ endowed with $C^r$-topology. Given a vector field $X\in\mathscr{X}^r(M)$. We denote by $\phi_t^X\colon M\rightarrow M$ the $C^r$-flow generated by $X$ and by $\Phi_t^X=D\phi_t^X$ the tangent map of $\phi_t^X$. 
When there is no confusion, we also denote them by $\phi_t$ and $\Phi_t$ for simplicity. 
We denote by $\sing(X)$ and $\per(X)$ respectively the set of singularities and periodic points respectively of $X$.
A point $p$ is \emph{a critical element} of $X$ if $p\in\sing(X)\cup\per(X)$.
Our Theorem~\ref{Thm:A-connecting} is a $C^r$-connecting lemma for geometric Lorenz attractors which creates homo/heteroclinic orbits associated to the singularity with a critical element (the singularity itself or a hyperbolic periodic orbit). 

\begin{theoremalph}
	\label{Thm:A-connecting}
	Let $r\in\mathbb{N}_{\geq2}\cup\{\infty\}$ and $X\in\mathscr{X}^r(M^3)$ which exhibits a geometric Lorenz attractor $\Lambda$ with singularity $\sigma\in\sing(\Lambda)$. Assume $z\in W^u(\sigma)$ and $p\in\Lambda$ is a critical element, then there exists $Y\in\mathscr{X}^r(M^3)$ which is arbitrarily $C^r$-close to $X$, such that
	$$
	\orb(z,\phi_t^Y)\subseteq W^u(\sigma,\phi_t^Y)\cap W^s(p,\phi_t^Y),
	$$  
	with $\orb(\sigma,\phi_t^X)=\orb(\sigma,\phi_t^Y)$ and $\orb(p,\phi_t^X)=\orb(p,\phi_t^Y)$.
\end{theoremalph}

\begin{remark}
(1)	A more detailed statement of Theorem~\ref{Thm:A-connecting} is Theorem~\ref{thm:Lorenz-connecting}.
	The definition (Definition \ref{def:Lorenz}) of geometric Lorenz attractor follows J. Guckenheimer and R. Williams \cite{guck,gw,williams}, since our Theorem \ref{Thm:A-connecting}  (\& Theorem \ref{thm:Lorenz-connecting}) requires that the reducing 1-dimensional map has derivative larger than $\sqrt{2}$. Following this definition, we show that the vector fields exhibiting a geometric Lorenz attractor is an open subset of $\mathscr{X}^r(M^3)$ for every $r\geq2$, see Proposition~\ref{prop:Lorenz}. 
	\medskip
	
(2) In the proof of Theorem~\ref{Thm:A-connecting} \&~\ref{thm:Lorenz-connecting}, we apply a beautiful observation by~\cite{am2} which shows that the induced stable foliation of geometric Lorenz attractors in the cross section is $C^1$-smooth. Moreover, we prove derivatives of holonomy maps of the induced stable foliations vary continuously  with respect to vector fields in $C^2$-topology. This implies the vector fields exhibiting a geometric Lorenz attractor forms a $C^2$-open subset in $\mathscr{X}^r(M^3)$ for every $r\in\mathbb{N}_{\geq 2}\cup\{\infty\}$, see Appendix \ref{Section:robust}.
\medskip

(3)  The critical element $p$ in Theorem~\ref{Thm:A-connecting} can be $\sigma$ itself or a periodic point. Theorem~\ref{Thm:A-connecting}  shows that for geometric Lorenz attractors,  a homoclinic or heteroclinic orbit associate to the singularity can be generated through a $C^r$-perturbation.
It has its own interest in the study of homoclinic bifurcations for geometric Lorenz attractors. For instance, we show that $C^r$-densely, the singularity inside a geometric Lorenz attractor admits a homoclinic orbit (Corollary~\ref{Cor:dense-loop}).
\end{remark}

\subsection{The space of ergodic measures on Lorenz attractors}
As an application of Theorem~\ref{Thm:A-connecting}, we study the space of invariant measures supported on geometric Lorenz attractors of $X\in\mathscr{X}^r(M^3)$ where  $r\in\mathbb{N}_{\geq2}\cup\{\infty\}$. 
This is motivated by the classical works of K. Sigmund~\cite{Sigmund-Axiom A,Sigmund-flow,Sigmund} on the invariant measure space of basic sets for Axiom A systems~\cite{smale} together with R. Ma\~n\'e's ergodic closing lemma~\cite{mane-ergodic closing}.
Given $X\in\mathscr{X}^r(M)$ and an invariant compact set $\Lambda$ of $\phi_t^X$, denote by $\mathcal{M}_{inv}(\Lambda)$ and $\mathcal{M}_{erg}(\Lambda)$ the spaces of invariant measures   and ergodic measures respectively of $X$ supported on $\Lambda$  endowed with the weak*-topology. 
Denote by $\mathcal{M}_{per}(\Lambda)$ the set of periodic measures (the invariant measure equidistributed on a single periodic orbit) supported on $\Lambda$. 

Recall that a basic set of an Axiom A diffeomorphism is a hyperbolic horseshoe which is topologically conjugate to a full shift, and in the case of Axiom A  vector fields, a basic set $\Lambda$ is the suspension of a hyperbolic horseshoe $\Lambda'$ with a continuous roof function. 
Moreover, every basic set is  a homoclinic class, where for a hyperbolic periodic point $p$, its homoclinic class $H(p)$ is the closure of the set of periodic points $q$ homoclinically related to $p$, i.e. $W^s(\orb(q))$ has non-empty transverse intersections with $W^u(\orb(p))$ and vice versa.
When $\Lambda$ is a basic set of an Axiom A system (diffeomorphism or vector field), K. Sigmund~\cite{Sigmund-Axiom A,Sigmund-flow} proved that  $\mathcal{M}_{per}(\Lambda)$ is dense in $\mathcal{M}_{inv}(\Lambda)$ and the set of ergodic measures with zero entropy forms a residual subset (i.e. a dense $G_\delta$ subset) of $\mathcal{M}_{inv}(\Lambda)$. 
Another work of K. Sigmund~\cite{Sigmund} shows that the space of ergodic measures over a basic set of a diffeomorphism is path connected (or arcwise connected). For Axiom A  vector fields, since a basic set $\Lambda$ is the suspension of a hyperbolic horseshoe $\Lambda'$ with a continuous roof function, thus by the one-to-one correspondence between invariant measures over $\Lambda$ and $\Lambda'$ (see for instance~\cite{parry-pollicott}), one has that $\mathcal{M}_{erg(\Lambda)}$ is also path connected. 
More statistical properties of periodic measures on basic sets are obtained by R. Bowen~\cite{Bowen-Axiom A,Bowen-Flow}.
In the non-hyperbolic case,  a recent work of A. Gorodetski-Y. Pesin~\cite{Go-Pe} proves that if an isolated homoclinic class  $H(p)$ of a diffeomorphism satisfies that all hyperbolic periodic orbits in $H(p)$ of the same index with $p$ are homoclinically related with each other, then $\mathcal{M}_{erg}(H(p))$ is path connected. This gives a criterion for path connectedness of the ergodic measure space restricted on homoclinic classes. Other results on connectedness and topological properties of the  invariant measure space for certain systems can be found in~\cite{bbg,GK}. 

Every geometric Lorenz attractor is a homoclinic class in which all periodic orbits are homoclinically related (see Proposition~\ref{prop:Lorenz}), which is similar as basic sets and the isolated homoclinic class where A. Gorodetski-Y. Pesin~\cite{Go-Pe} consider. However, due to the existence of the singularity, the ergodic measure space of a geometric Lorenz attractor may display a quite different nature.
We show that  for $C^r(r\geq 2)$-dense vector fields  $X\in\mathscr{X}^r(M^3)$ admitting a geometric Lorenz attractor $\Lambda$, the singular measure $\delta_\sigma$ is isolated in the space of ergodic measures $\mathcal{M}_{erg}(\Lambda)$ which breaks the connectedness of $\mathcal{M}_{erg}(\Lambda)$.
The key point is by applying Theorem~\ref{Thm:A-connecting} to make the two branches of the unstable manifold of $\sigma$ lie on the stable manifold of a periodic orbit $\orb(p)$ through $C^r$-perturbation.
Then a criterion for isolated ergodic measures (Proposition~\ref{Prop:kick-out singularity}) that we will prove  ensures that the singular measure of the target system is isolated.

Meanwhile, as a sharp comparison of the above dense part, we prove that the denseness of periodic measures inside the space of invariant measures and the path connectedness of the ergodic measure space hold for geometric Lorenz attractors of $C^r$-generic $X\in \mathscr{X}^r(M^3)$ (i.e. vector fields in a residual (dense $G_\delta$)subset of $\mathscr{X}^r(M^3)$), which proves ~\cite[Conjecture 2]{b-survey} for geometric Lorenz attractors in $C^r(r\geq 2)$-topology. This is in spirit of Sigmund's work~\cite{Sigmund-Axiom A,Sigmund-flow,Sigmund}  mentioned above and R. Ma\~n\'e's ergodic closing lemma~\cite{mane-ergodic closing}.
The ergodic closing lemma is a milestone  in dynamical systems and ergodic theory, and it 
plays a key role in solving the stability conjecture.
A classical consequence of the ergodic closing lemma is that for $C^1$-generic diffeomorphisms or vector fields , the set of periodic measures and singular measures are dense in the space of ergodic measures.  
However, things become more complicated to verify the denseness of periodic measures in the space of ergodic measures when restricted to a non-hyperbolic compact  invariant set $\Lambda$. 
The difficulty is how to guarantee that  the periodic measures (maybe obtained by perturbations) provided by the ergodic closing lemma are contained in $\Lambda$, see for instance C. Bonatti's  survey~\cite[Conjecture 2]{b-survey}.
Progresses and some important criterion can be found in the work of F. Abdenur-C. Bonatti-S. Crovisier~\cite{abc}. 
Other works on the approximation of invariant measures by periodic ones can be found in~\cite{btv,bonatti-zhang,dgmr,hirayama,lls,Yang-Zhang}.

In sum, our Theorem~\ref{Thm:B} gives a sharp contrast on ergodic measure space for geometric Lorenz attractors between $C^r (r\geq 2)$-dense and $C^r$-generic vector fields in $\mathscr{X}^r(M^3)$.

\begin{theoremalph}\label{Thm:B}
	For every $r\in\mathbb{N}_{\geq2}\cup\{\infty\}$, there exist a dense subset $\mathscr{D}^r\subset\mathscr{X}^r(M^3)$ and a residual subset $\mathscr{R}^r\subset\mathscr{X}^r(M^3)$, such that
	\begin{itemize}
		
		\item If $\Lambda$ is  a geometric Lorenz attractor of $X\in\mathscr{D}^r$ , then 
		$\overline{\mathcal{M}_{per}(\Lambda)}\subsetneqq \overline{\mathcal{M}_{erg}(\Lambda)}\subsetneqq \mathcal{M}_{inv}(\Lambda)$. 
		Moreover, the space $\mathcal{M}_{erg}(\Lambda)$ is not connected.
		
		\item If $\Lambda$ is  a geometric Lorenz attractor of $X\in\mathscr{R}^r$, then  
		$\overline{\mathcal{M}_{per}(\Lambda)}=
		\overline{\mathcal{M}_{erg}(\Lambda)}=		\mathcal{M}_{inv}(\Lambda)$. 
		Moreover, the space $\mathcal{M}_{erg}(\Lambda)$ is path connected.
	\end{itemize}
\end{theoremalph}

\begin{remark}
	Recall that the geometric Lorenz attractor is a homoclinic class, see Proposition~\ref{prop:Lorenz}. To the authors' knowledge, the dense part of Theorem~\ref{Thm:B} (together with Theorem~\ref{Thm:SH-attractor} in Appendix~\ref{Section:C1}) is the first known example of locally dense systems (in any $C^r$-topology where $r\geq 1$) exhibiting an isolated ergodic measure supported on a homoclinic class.
\end{remark}

\begin{remark}
	    In our construction for the dense part of Theorem~\ref{Thm:B}, the atomic measure $\delta_{\sigma}$ of the singularity $\sigma$ is isolated in the ergodic measure space. We kick the atomic measure $\delta_{\sigma}$ out in the same spirit of \cite{bcgp}, where the authors kick out a single orbit of a chain recurrence class from a homoclinic class.
	    Actually, for every geometric Lorenz attractor $\Lambda$,  let $\mathcal{M}_0=\mathcal{M}_{erg}(\Lambda)\setminus \{\delta_{\sigma}\}$ and $\mathcal{M}_1=\{\mu\in \mathcal{M}_{inv}(\Lambda): \mu(\sigma)=0 \}$, then $\mathcal{M}_0$ is path connected and moreover, $\overline{\mathcal{M}_{per}(\Lambda)}=\overline{\emph{Convex}(\mathcal{M}_0)}=\overline{\mathcal{M}_1}$. Here $\emph{Convex}(\mathcal{M}_0)$ consists all elements that are convex combinations of measures in $\mathcal{M}_0$. 
\end{remark}

The following corollary of Theorem~\ref{Thm:B} concerns the support and entropy of invariant measures for $C^r(r\geq 2)$-generic geometric Lorenz attractors.

\begin{corollary}\label{Cor:Cr-entropy-support}
	For every $r\in\mathbb{N}_{\geq2}\cup\{\infty\}$, there exists a residual subset $\mathscr{R}^r\subset\mathscr{X}^r(M^3)$, such that for every $X\in\mathscr{R}^r$, if $\Lambda$ is a geometric Lorenz attractor of $X$, then 
	\begin{itemize}
		\item there exists a residual subset $\mathcal{M}_{res}$ in $\mathcal{M}_{inv}(\Lambda)$ such that every $\mu\in\mathcal{M}_{res}$ is ergodic with $\supp(\mu)=\Lambda$ and $h_{\mu}=0$.
		\item there exists a dense subset $\mathcal{M}_{den}$ in $\mathcal{M}_{inv}(\Lambda)$ such that every $\mu\in\mathcal{M}_{den}$ is ergodic with $h_{\mu}>0$.
	\end{itemize}
\end{corollary}

\begin{remark}
	For every geometric Lorenz attractor $\Lambda$ of every $X\in \mathscr{X}^r(M^3)$, the measures with zero entropy and full support  form a residual subset in $\mathcal{M}_{inv}(\Lambda)$. See  Remark~\ref{Rem:entropy and support}.
\end{remark}

\begin{remark}
For vector fields of any dimension in $C^1$-topology, we obtain similar conclusions as in Theorem~\ref{Thm:B} and Corollary~\ref{Cor:Cr-entropy-support} for a class of  singular hyperbolic attractors, see Appendix~\ref{Section:C1}. 

\end{remark}

\subsection*{Organization of the paper} In Section~\ref{Section:Pre}, we state some preliminaries on singular hyperbolicity and geometric Lorenz attractors.
The proof of Theorem~\ref{Thm:A-connecting} together with a detailed characterization of geometric Lorenz attractors is given in Section~\ref{Section:Cr-Lorenz}.
Then we study the ergodic measure space of geometric Lorenz attractors and prove Theorem~\ref{Thm:B} together with Corollary~\ref{Cor:Cr-entropy-support} in Section~\ref{Section:measure}.
In Appendix~\ref{Section:robust}, we give a proof of Proposition~\ref{prop:Lorenz} which states that the $C^r(r\geq 2)$-vector fields admitting a geometric Lorenz attractor forms an open subset of $\mathscr{X}^r(M)$. In Appendix~\ref{Section:C1}, we sketch some conclusions concerning ergodic measure spaces of singular hyperbolic attractors for $C^1$ vector fields of any dimension.


\subsection*{Acknowledgment} We would like to thank S. Crovisier, S. Gan, G. Liao, C. A. Morales, P. Varandas, L. Wen, D. Yang, J. Yang and J. Zhang for useful suggestions and comments. We are very grateful to the anonymous referee of our paper for her/his valuable comments, which help us a lot to improve the paper. This work was prepared  during Y. Shi and X. Wang visiting School of Mathematical Sciences, Fudan University. The two authors would like to thank their hospitality.

\section{Preliminary}\label{Section:Pre}
   Recall that for a $C^r(r\geq 1)$-vector field $X\in\mathscr{X}^r(M)$, we denote by $\phi_t^X\colon M\rightarrow M$ the $C^r$-flow generated by $X$ and by $\Phi_t^X=D\phi_t^X$ the tangent map of $\phi_t^X$. We also use $\phi_t$ and $\Phi_t$ for simplicity if there is no confusion.
   
\subsection{Singular hyperbolicity}

We give the definition of hyperbolic and singular hyperbolic set.

\begin{definition}\label{Def:hyp}   
    Given  a vector field $X\in\mathscr{X}^1(M)$. An invariant compact set $\Lambda$ is {\it hyperbolic} if $\Lambda$ admits a continuous $\Phi_t$-invariant splitting $T_{\Lambda}M=E^s\oplus \langle X\rangle\oplus E^u$, where $\langle X\rangle$ denotes the one-dimensional linear space generated by the flow direction, such that $E^s$ is contracted by $\Phi_t$ and $E^u$ is expanded by $\Phi_t$. To be precise, there exist two constants $C>0$ and $\eta>0$, such that for any $x\in\Lambda$ and any $t\geq 0$, it satisfies:
    \begin{itemize}
    	\item for any $v\in E^s(x)$, $\|\Phi_t(v)\|\leq Ce^{-\eta t}\|v\|$;
    	\item for any $v\in E^u(x)$, $\|\Phi_{-t}(v)\|\leq Ce^{-\eta t}\|v\|$.
    \end{itemize}
The stable dimension $\dim(E^s)$ is called the {\it index} of the hyperbolic splitting.
\end{definition}

The notion of singular hyperbolicity was  introduced by Morales-Pacifico-Pujals~\cite{mpp} to describe the geometric structure of Lorenz attractors and was then developed in higher dimensions by~\cite{lgw,Me-Mo}.
It is proved in~\cite{mpp}  that any robustly transitive set of a three dimensional vector field is singular hyperbolic and similar conclusion holds for higher dimensions~\cite{lgw,zgw}.  We refer to \cite{3-dim flow} for a comprehensive introduction of singular hyperbolic vector fields. 

\begin{definition}\label{Def:singular-hyp}
	Given  a vector field $X\in\mathscr{X}^1(M)$. An invariant compact set $\Lambda$ is {\it singular hyperbolic} if $\Lambda$ admits a continuous $\Phi_t$-invariant splitting $T_{\Lambda}M=E^{ss}\oplus E^{cu}$ and there exist two constants $C,\eta>0$, such that for any $x\in\Lambda$ and any $t\geq 0$, it satisfies:
	\begin{itemize}
		\item $E^{ss}\oplus E^{cu}$ is a {\it dominated splitting}: $\|\Phi_t|_{E^{ss}(x)}\|\cdot \|\Phi_{-t}|_{E^{cu}(\phi_t(x))}\|< Ce^{-\eta t}$;
		\item $E^{ss}$ is contracted by $\Phi_t$: $\|\Phi_t(v)\|< Ce^{-\eta t}$ for any $v\in E^{ss}(x)$;
		\item $E^{cu}$ is sectionally expanded by $\Phi_t$:  $\|\det\Phi_t(x)|_{V_x}\|> Ce^{\eta t}$ for any 2-dimensional  subspace $V_x\subset E^{cu}_x$.
	\end{itemize}

\end{definition}

\begin{remark}\label{Rem:singular hyperbolicity}
	(1) Given $X\in\mathscr{X}^1(M)$. If an invariant compact set $\Lambda$ is hyperbolic, then $\Lambda$ must contain no singularity. On the other hand, if $\sing(\Lambda)=\emptyset$, then $\Lambda$ is hyperbolic if and only if $\Lambda$ is singular hyperbolic for $X$ or for $-X$.
	
	(2) By definition, singular hyperbolicity is a robust property. To be precise, if $\Lambda$ is a singular hyperbolic invariant compact set of $X\in\mathscr{X}^1(M)$ associated with splitting $T_{\Lambda}M=E^{ss}\oplus E^{cu}$ and constants $(C,\eta)$, then there exist a neighborhood $U$ of $\Lambda$ and a neighborhood $\mathcal{U}\subset \mathscr{X}^1(M)$ of $X$ such that for any $Y\in\mathcal{U}$, the maximal invariant set of $\phi_t^Y$ in $U$ is singular hyperbolic associated with the same stable dimension and constants $(C,\eta)$.
	
\end{remark}

\subsection{The geometric Lorenz attractor}

We follow the constructions of Guckenheimer and Williams~\cite{guck,gw,williams} to give the definition of {\it geometric Lorenz attractor} for any $C^r (r\geq 1)$ vector field. 
Let $X\in\mathscr{X}^r(M)$ where $r\in\mathbb{N}\cup\{\infty\}$,  an open subset $U\subseteq M$ is called an \emph{attracting region} of $X$, if for every $x\in\partial U$, the vector field $X(x)$ is transverse to $\partial U$ and points to the interior of $U$, and $\phi_t^X(x)\in U$ for every $t>0$. 
This implies $\overline{\phi_t^X(U)}=\phi_t^X(\overline{U})\subsetneqq U$ for every $t>0$. 
An invariant compact set $\Lambda$ is an {\it attractor} of $X$ if $\Lambda$ is transitive and $\Lambda$ is the maximal invariant set of an attracting region.
Recall that $M^3$ denotes a three dimensional closed smooth Riemannian manifold.

\begin{definition}\label{def:Lorenz}
	Let $r\geq 2$ and $X\in\mathscr{X}^r(M^3)$. We say $X$ admits a \emph{geometric Lorenz attractor} $\Lambda$, if $X$ has an attracting region $U\subset M^3$ such that  $\Lambda=\bigcap_{t>0}\phi^X_t(U)$ is a singular hyperbolic attractor and satisfies the following properties:
	\begin{itemize}
		\item $\Lambda$ contains a unique singularity $\sigma$ with three eigenvalues $\lambda_1<\lambda_2<0<\lambda_3$ with respect to $\Phi_1(\sigma)$ satisfying $\lambda_1+\lambda_3<0$ and $\lambda_2+\lambda_3>0$.  
		\item $\Lambda$ admits a $C^r$-smooth cross section $\Sigma$ which is $C^1$-diffeomorphic to $[-1,1]\times[-1,1]$, such that for every $z\in U\setminus W^s_{\it loc}(\sigma)$, there exists $t>0$ such that $\phi_t^X(z)\in\Sigma$, and $l=\{0\}\times[-1,1]=W^s_{\it loc}(\sigma)\cap\Sigma$.
		\item The Poincar\'e map $P:\Sigma\setminus l\rightarrow\Sigma$ is $C^1$-smooth in the coordinate $\Sigma=[-1,1]^2$ and has the form
		$P(x,y)=\big( f(x)~,~H(x,y) \big)$ for every 
		$(x,y)\in[-1,1]^2\setminus l$.  
		Moreover, it satisfies
		\begin{itemize}
			\item $H(x,y)<0$ for $x>0$, $H(x,y)>0$ for $x<0$, and $\sup_{(x,y)\in\Sigma\setminus l}\big|\partial H(x,y)/\partial y\big|<1$, 
			$\sup_{(x,y)\in\Sigma\setminus l}\big|\partial H(x,y)/\partial x\big|<1$;
			\item the one-dimensional quotient map $f:[-1,1]\setminus\{0\}\rightarrow[-1,1]$ is $C^1$-smooth and satisfies
			$\lim_{x\rightarrow0^-}f(x)=1$,
			$\lim_{x\rightarrow0^+}f(x)=-1$, $-1<f(x)<1$ and
			$f'(x)>\sqrt{2}$ for every $x\in[-1,1]\setminus\{0\}$.
		\end{itemize}
	\end{itemize}
\end{definition}

\begin{figure}[htbp]
	\centering
	\includegraphics[width=14cm]{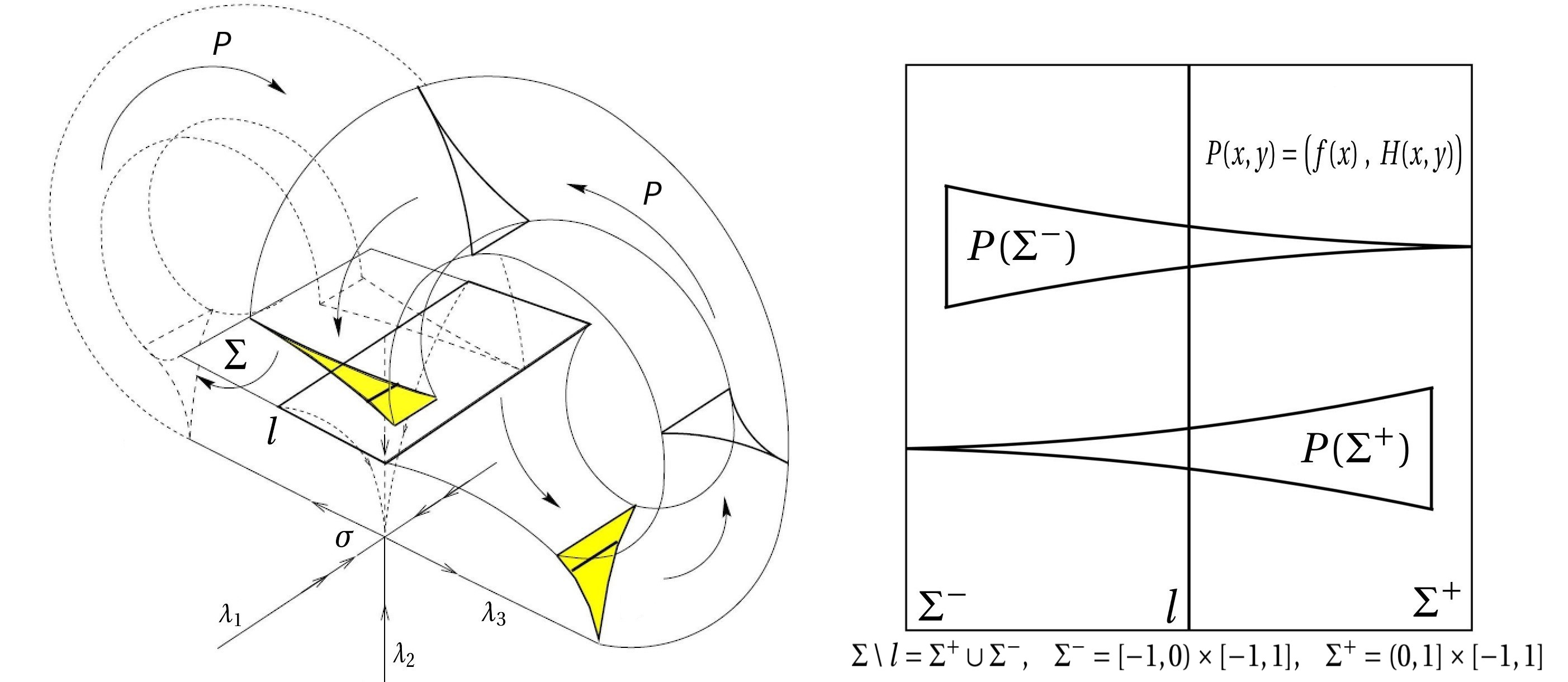}
	\caption{Geometric Lorenz attractor and return map}
\end{figure}

\begin{remark}\label{rk:unstable}
	The unstable manifold of $\sigma$ contains two orbits. We denote $z^+,z^-\in W^u(\sigma)\cap\Sigma$ which satisfy $\phi_t^X(z^{\pm})\notin\Sigma$ for every $t<0$, and
	$$
	z^+=(-1,y^+)\in\Sigma, \qquad {\rm and} \qquad
	z^-=(1,y^-)\in\Sigma.
	$$
	Here we can assume $-1<y^+<0$ and $0<y^-<1$ by changing the $y$-coordinate.
	For every $(x,y)\in\Sigma\setminus l$, if $x\rightarrow0^+$, then $P(x,y)$ approaches $z^+$;  if $x\rightarrow0^-$, then $P(x,y)$ approaches $z^-$. 
\end{remark}

\begin{lemma}\label{lem:cone}
	The Poincar\'e map $P(x,y)=\big( f(x),H(x,y) \big):\Sigma\setminus l\rightarrow\Sigma$ satisfies
	$$
	\lim_{x\rightarrow0}f'(x)=+\infty,
	\qquad {\rm and} \qquad
	\lim_{x\rightarrow0}\big|\partial H(x,y)/\partial y\big|=0.
	$$
	
	Moreover, for every $\alpha>0$, let the cone field $\cC_{\alpha}=\bigcup_{(x,y)\in\Sigma}\cC_{\alpha}(x,y)$ be defined on $\Sigma$ as
	$$
	\cC_{\alpha}(x,y)~=~\big\{~v=a\cdot\partial/\partial x 
	+b\cdot\partial/\partial y:~
	|b|\leq\alpha\cdot |a|~\big\}.
	$$
	Then for every $\alpha\geq1/(\sqrt{2}-1)$, the Poincar\'e map $P$ satisfies 
	$$
	DP\big(\cC_{\alpha}(x,y)\big)\subseteq\cC_{\alpha}(P(x,y)),
	\qquad \forall (x,y)\in\Sigma\setminus l.
	$$
\end{lemma}

\begin{proof}
	Since $(x,y)\rightarrow W^s_{\it loc}(\sigma)$ as $x\rightarrow0$, the orbit of $(x,y)$ approaches the singularity $\sigma$. As we assumed $\lambda_1<\lambda_2<0$ and $\lambda_2+\lambda_3>0$, thus $\lim\limits_{x\rightarrow0}f'(x)=+\infty$ and $\lim\limits_{x\rightarrow0}\big|\partial H(x,y)/\partial y\big|=0$. Moreover, for every $\alpha\geq1/(\sqrt{2}-1)$ and $v=\partial/\partial x+\alpha\cdot\partial/\partial y\in\cC_{\alpha}$, if we denote  $DP(v)=a\cdot\partial/\partial x+b\cdot\partial/\partial y$, then we have $|a|>\sqrt{2}$ and $|b|\leq 1+\alpha\leq \alpha\cdot\sqrt{2}$. This implies $DP(v)\in\cC_{\alpha}$.
\end{proof}

\begin{remark}\label{rk:H}
	For every point $z\in\Sigma\cap\Lambda$, the cone field $\cC_{\alpha}(z)\oplus\left\langle X(z)\right\rangle$ is a neighborhood of center-unstable bundle $E^{cu}(z)$ of $z$.
	
	In Definition \ref{def:Lorenz}, the property $\sup_{(x,y)\in\Sigma\setminus l}\big|\partial H(x,y)/\partial x\big|<1$ is not necessary. We only need there exists some constant $K>0$, such that $\sup_{(x,y)\in\Sigma\setminus l}\big|\partial H(x,y)/\partial x\big|<K$, which is automatically holds because $\Lambda$ is singular hyperbolic and $H$ is uniformly $C^1$-smooth. Then we can take $\alpha\geq K/(\sqrt{2}-1)$ and the cone field $\cC_{\alpha}$ is $DP$-invariant.
\end{remark}

It has been showed in~\cite{c1lorenz} that, following Definition~\ref{def:Lorenz}, every geometric Lorenz attractor $\Lambda$ is robustly transitive.
In Section \ref{Section:Cr-Lorenz}, we will introduce a series of properties of geometric Lorenz attractors. In particular, the $C^r$-vector fields exhibiting a geometric Lorenz attractor consists an open subset in $\mathscr{X}^r(M^3)$ for every $r\geq2$ (Proposition \ref{prop:Lorenz}).

\section{A $C^r$-connecting lemma of the geometric Lorenz attractors}\label{Section:Cr-Lorenz}

In this section, we give the proof of Theorem~\ref{Thm:A-connecting}. The more detailed statement is Theorem~\ref{thm:Lorenz-connecting} below. Before the proof, we first state some basic properties of geometric Lorenz attractors in the following subsection.

\subsection{Properties of geometric Lorenz attractors}\label{Section:properties-Lorenz}

In this subsection, we investigate some properties for the geometric Lorenz attractor. We assume $r\in\mathbb{N}_{\geq2}\cup\{\infty\}$ and $X\in\cX^r(M^3)$ which exhibits a geometric Lorenz attractor $\Lambda$ with cross section $\Sigma$. Let $P(x,y)=\big(f(x),H(x,y)\big):\Sigma\setminus l\rightarrow\Sigma$ be the Poincar\'e map of $\phi_t^X$. 

Recall that $\Lambda$ is a singular hyperbolic attractor with splitting $T_{\Lambda}M^3=E^{ss}\oplus E^{cu}$ and there exists a H\"older continuous foliation $\cF^{ss}$ tangent to $E^{ss}$ everywhere~\cite{psw}.
Since the strong stable bundle $E^{ss}$ is determined by the positive orbit of a point and $\Lambda$ is an attractor, $E^{ss}$ and $\cF^{ss}$ is well defined in the whole attracting region $U$. Moreover, we have the following lemma.

\begin{lemma}\label{lem:stable}
	For every $z\in U\setminus\cF^{ss}(\sigma)$, if we define 
	$$
	\cF^s(z)=\big(\bigcup_{t\in\RR}\phi_t^X(\cF^{ss}(z))\big)\cap U,
	$$
	then $\cF^s$ is a 2-dimensional $\phi_t^X$-invariant foliation in $U\setminus\cF^{ss}(\sigma)$. Moreover, each leaf of $\cF^s$ is $C^r$-smooth and the holonomy maps of $\cF^s$ is $C^{1}$-smooth.
\end{lemma}

\begin{proof}
	The definition of $\cF^s$ implies it is a 2-dimensional $\phi_t^X$-invariant foliation in $U\setminus\cF^{ss}(\sigma)$. The classical invariant manifold theorem \cite{hps} shows each leaf of $\cF^{ss}$ is $C^r$-smooth, so does each leaf of $\cF^s$. The holonomy maps of $\cF^s$ is $C^{1}$-smooth\footnote{In fact, \cite[Lemma 7.1]{am2} proves more that the holonomy maps of $\cF^s$ is $C^{1+}$-smooth.} has been proved in \cite[Lemma 7.1]{am2}.
\end{proof}

\begin{lemma}\label{lem:big-section}
	There exists a $C^r$-smooth cross section $\Sigma_1$ of $\phi_t^X$ which is an extension of $\Sigma$ and $C^1$-diffeomorphic to $[-1-\epsilon,1+\epsilon]^2$ for some $\epsilon>0$ such that
	$$
	l_1=\{0\}\times[-1-\epsilon,1+\epsilon]
	=\Sigma_1\cap W^s_{\it loc}(\sigma,\phi_t^X),
	\quad {\rm and} \quad
	\Sigma=[-1,1]^2\subset
	[-1-\epsilon,1+\epsilon]^2=\Sigma_1.
	$$
	The Poincar\'e map $P:\Sigma\setminus l\rightarrow\Sigma$ smoothly extends to a Poincar\'e map
	$P_1:\Sigma_1\setminus l_1\rightarrow\Sigma_1$, which has the form $P_1(x,y)=\big( f_1(x), H_1(x,y) \big)$ for every $(x,y)\in\Sigma\setminus l_1$ and satisfies 
	$P_1|_{\Sigma\setminus l}\equiv P$.
	Moreover, they satisfy:
	\begin{itemize}
		\item $H_1(x,y)>0$ for $x>0$, $H_1(x,y)<0$ for $x<0$, and $\sup_{(x,y)\in\Sigma_1\setminus l_1}\big|\partial H_1(x,y)/\partial y\big|<1$,
		$\sup_{(x,y)\in\Sigma\setminus l}\big|\partial H(x,y)/\partial x\big|<1$;
		\item the one-dimensional quotient map $f_1:[-1-\epsilon,1+\epsilon]\setminus\{0\}\rightarrow[-1,1]$ is $C^1$-smooth and satisfies $f_1|_{[-1,1]\setminus\{0\}}\equiv f$, $-1<f(x)<1$ and
		$f'(x)>\sqrt{2}$ for every $x\in[-1-\epsilon,1+\epsilon]\setminus\{0\}$.
	\end{itemize}
	In particular, the image of the Poincar\'e map $P_1$ satisfies $\overline{P_1(\Sigma_1\setminus l_1)} \subseteq [-1,1]\times(-1,1)$.
\end{lemma}

\begin{remark}
	Here even the cross section $\Sigma_1$ is $C^r$-smooth, the stable foliation $\cF^s\cap\Sigma_1$ induced by the flow intersecting with $\Sigma_1$ can only be $C^1$ but not $C^r$-smooth. If we want to take a coordinate $[-1-\epsilon,1+\epsilon]^2$ where $x\times[-1-\epsilon,1+\epsilon]$ is the leaf of induced stable foliations $\cF^s\cap\Sigma_1$, this coordinate can only be $C^1$-smooth. This is why we say $\Sigma_1$ is $C^1$-diffeomorphic to $[-1-\epsilon,1+\epsilon]^2$ in the lemma.
\end{remark}

\begin{proof}
	Since $\Sigma$ is a $C^r$-smooth cross section of $\phi_t^X$, we can extend $\Sigma$ to a larger surface $\Sigma_1$ transverse to $\phi_t^X$. 
	Moreover, Lemma \ref{lem:stable} shows that the stable foliation $\cF^s$ intersects $\Sigma_1$ induced a $C^{1}$-smooth foliation in $\Sigma_1$, where we already have $\cF^s\cap\Sigma=\big\{x\times[-1,1]:x\in[-1,1]\big\}$. So we can  $C^1$-smoothly extend the coordinate of $\Sigma$ to $\Sigma_1=[-1-\epsilon,1+\epsilon]^2$ for some $\epsilon>0$, such that
	$$
	\cF^s\cap\Sigma_1~=~
	\big\{x\times[-1-\epsilon,1+\epsilon]:~
	x\in[-1-\epsilon,1+\epsilon]\big\},
	$$
	and $l_1=0\times[-1-\epsilon,1+\epsilon]=W^s_{\it loc}(\sigma,\phi_t^X)\cap\Sigma_1$. 
	
	Let $P_1:\Sigma_1\setminus l_1\rightarrow \Sigma_1$ be the Poincar\'e map associated to $\Sigma_1$, then it has the form $P_1(x,y)=(f_1(x), H_1(x,y))$ and $P_1|_{\Sigma\setminus l}\equiv P$. In particular, $f_1$ is $C^1$-smooth and $f_1|_{[-1,1]\setminus\{0\}}\equiv f$.
	From the definition of geometric Lorenz map, the one-dimensional quotient map $f:[-1,1]\setminus\{0\}\rightarrow[-1,1]$ satisfies $-1<f(x)<1$ and $f'(x)>\sqrt{2}$ everywhere. This implies $-1<f(-1)<0<f(1)<1$. Thus if we choose $\epsilon$ small enough, we also have $-1<f_1(1-\epsilon)<0<f(1+\epsilon)<1$. Shrinking $\epsilon$ if necessary, $P_1$ is well defined on $\Sigma_1\setminus l_1$, and $H_1,f_1$ satisfies all estimations of $H,f$. In particular, we have 
	$$
	\overline{P_1(\Sigma_1\setminus l_1)}=
	P_1(\Sigma_1\setminus l_1)\cup\{z^+,z^-\}
	\subseteq[-1,1]\times(-1,1).
	$$
\end{proof}

\begin{notation}\label{notat:section}
	From now on, we use $\Sigma,P,f,H$ to denote $\Sigma_1,P_1,f_1,H_1$ respectively for the simplicity of notations. We denote $\pi_x:\Sigma\rightarrow[-1-\epsilon,1+\epsilon]$ the projection to $x$-coordinate on $\Sigma$.
	Since $\lim_{x\rightarrow0}f'(x)=+\infty$, there exists $\lambda_0>\sqrt{2}$, such that $f'(x)>\lambda_0$ for every $x\in[-1-\epsilon,1+\epsilon]\setminus\{0\}$. 
	On the other hand, since $f(-1)>-1$, we must have $\lambda_0<2$.
	
	We fix $\alpha=1/(\sqrt{2}-1)$. The proof of Lemma \ref{lem:cone} shows that the cone field $\cC_{\alpha}$ satisfies
	$$
	DP(\cC_{\alpha}(x,y))~\subseteq~
	\cC_{\sqrt{2}\alpha/\lambda_0}(P(x,y)),
	\qquad \forall (x,y)\in\Sigma\setminus l.
	$$
\end{notation}

\begin{definition}\label{def:map}
	Let $\mathscr{L}^r~(r\geq 1)$ be the set consisting of $C^r$-maps $f:[-1,1]\setminus\{0\}\rightarrow(-1,1)$ which satisfies
	\begin{itemize}
		\item $\lim_{x\rightarrow0^-}f(x)=1$, and $\lim_{x\rightarrow0^-}f'(x)=+\infty$;
		\item $\lim_{x\rightarrow0^+}f(x)=-1$, and $\lim_{x\rightarrow0^+}f'(x)=+\infty$;
		\item $-1<f(x)<1$ and $f'(x)>\sqrt{2}$
		for every $x\in[-1,1]\setminus\{0\}$.
	\end{itemize}
	We call $f\in\mathscr{L}^r$ a \emph{Lorenz expanding map}.
\end{definition}

\begin{figure}[htbp]
	\centering
	\includegraphics[width=7cm]{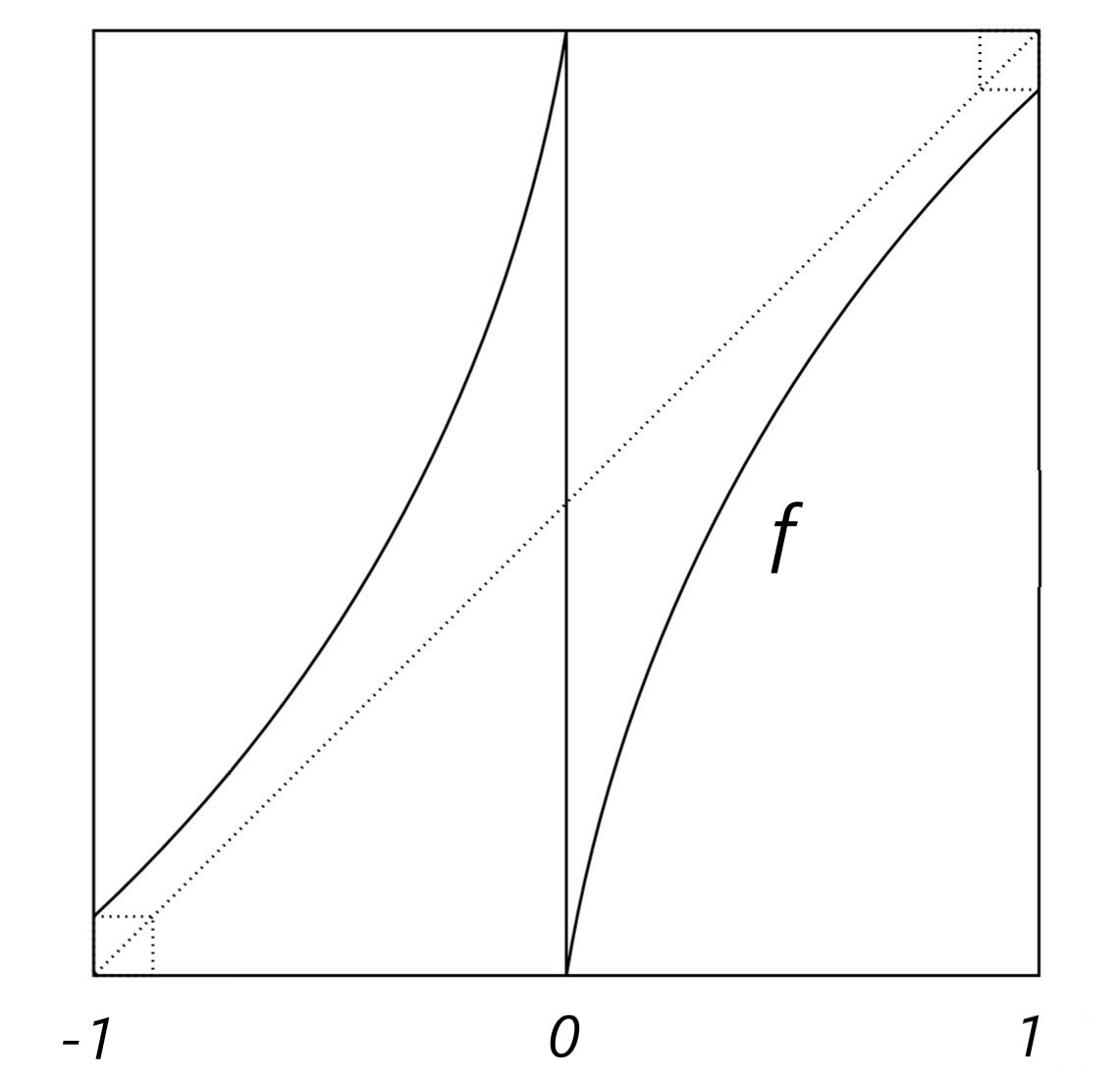}
	\caption{The Lorenz expanding map}
\end{figure}

The following lemma is standard in the study of geometric Lorenz attractors. We include its proof which illustrates the main idea of proving Claim \ref{clm:cover} in the proof of Theorem \ref{thm:Lorenz-connecting}.

\begin{lemma}\label{lem:evenually-onto}
	Let $f\in\mathscr{L}^r$, then $f$ is eventually onto, i.e. for every open interval $J\subset[-1,1]\setminus\{0\}$, there exists an integer $N>0$, such that $\bigcup_{i=0}^Nf^i(J)=(-1,1)$.
\end{lemma}

\begin{proof}[New proof of Lemma \ref{lem:evenually-onto}]
	For every $n\geq 1$, we consider the map $f^n:J\rightarrow[-1,1]$. 
	Let $$D_n=\{t\in J: f^m(t)=0, \text{ for some $1\leq m\leq n-1$}\}~~\text{for $n\geq 2$ and } D_1=\emptyset.$$
	We call $t\in D_n$ a discontinuity point of $f^n$. Then $D_n$ is a finite subset of $J$ with $D_n\subseteq D_{n+1}$, and $f^n$ is well-defined and $C^1$-smooth on each connected component of $J\setminus D_n$. Moreover, if $(a,b)\subseteq J$ is a connected component of $J\setminus D_n$, then $f^n|_{(a,b)}$ is a diffeomorphism from $(a,b)$ to its image 
	$$
	f^n(a,b)~=~
	\big(\lim_{x\rightarrow a^+}f^n(x)~,~
	\lim_{x\rightarrow b^-}f^n(x)\big)
	\quad {\rm  and} \quad
	(f^n)'(x)\geq \lambda_0^n, \quad \forall x\in(a,b).
	$$

	\begin{claim}\label{clm:counting}
		For every $n\geq1$, assume $I=(a,b)$ is a connected component of 
		$J\setminus D_n$, then either $\#\big(I\cap D_{n+2}\big)\leq1$, or 
		$(-1,1)\subseteq f^{n+1}(I)\cup f^{n+2}(I)$.
	\end{claim}
	\begin{proof}[Proof of the claim]
		Recall that $f^n$ is well-defined and $C^1$-smooth on $I$. Let $(c,d)=f^n(I)$. Since $I\cap D_n=\emptyset$, one has that $f^m(t)\neq 0$ for any $t\in I$ and $1\leq m\leq n-1$ by definition of $D_n$. If $0\notin(c,d)$, which means that $f^n(t)\neq 0$ for any $t\in I$, then one has $I\cap D_{n+1}=\emptyset$ by definition of $D_{n+1}$.  As a consequence $f^{n+1}(I)=f(c,d)$ is an open interval in $(-1,1)$. Thus there exists at most one point $t\in I$, such that $f^{n+1}(t)=0$. This implies $\#\big(I\cap D_{n+2}\big)\leq1$.
		
		If on the other hand $0\in(c,d)=f^n(I)$, then we denote $t_0\in I$ satisfying $f^n(t_0)=0$. Then the two open intervals $(a,t_0)$ and $(t_0,b)$ are two connected components of $J\setminus D_{n+1}$. Moreover, their images by $f^{n+1}$ satisfies
		$$
		f^{n+1}(a,t_0)=\big(\lim_{x\rightarrow a^+}f(x),~1\big),
		\qquad {\rm and } \qquad
		f^{n+1}(t_0,b)=\big(-1,~\lim_{x\rightarrow a^+}f(x)\big).	
		$$
		If neither of them contains $0$, then we have $\#\big(I\cap D_{n+2}\big)=\#\big(I\cap D_{n+1}\big)=1$.
		
		Otherwise, we have either  $0\in f^{n+1}(a,t_0)$ or $0\in f^{n+1}(t_0,b)$. Thus either 
		$$
		[0,1)\subset f^{n+1}(a,t_0)\subset f^{n+1}(I),
		\qquad {\rm or} \qquad
		(-1,0]\subset f^{n+1}(t_0,b)\subset f^{n+1}(I).
		$$
		In both cases, iterating by $f$ again, we have 
		$(-1,1)\subseteq f^{n+1}(I)\cup f^{n+2}(I)$.
	\end{proof}
	
	Now we show that there exists some $n>0$ and a connected component $I$ of $J\setminus D_n$, such that $(-1,1)\subseteq f^{n+1}(I)\cup f^{n+2}(I)$. Otherwise, Claim \ref{clm:counting} implies
	$$
	\# D_2\leq1, \qquad 
	\# D_4\leq\# D_2+\# D_2+1\leq 3,
	\qquad \cdots \qquad 
	\# D_{2n}\leq \# D_{2n-2}+\# D_{2n-2}+1\leq 2^n-1.
	$$
	This implies $J\setminus D_{2n}$ has at most $2^n$ connected components. So there exists a connected component $I_{2n}\subset J\setminus D_{2n}$, such that its length satisfies $|I_{2n}|\geq 2^{-n}|J|$.
	
	Since $f$ is a diffeomorphism on each connected component, and satisfies $(f^{2n})'>\lambda_0^{2n}$, thus the length of the interval $f^{2n}(I_{2n})$ satisfies
	$$
	\left|f^{2n}(I_{2n})\right| 
	~\geq~ \lambda_0^{2n}\cdot|I_{2n}|
	~\geq~ \left(\frac{\lambda_0^2}{2}\right)^n\cdot|J|
	~\rightarrow~+\infty,
	\qquad {\rm as}\quad n\rightarrow+\infty.
	$$
	This is absurd since $f^{2n}(I_{2n})\subseteq(-1,1)$. We take $N=n+2$ finishing the proof of this lemma.
\end{proof}

The following proposition shows that every Lorenz attractor is a homoclinic class and every pair of periodic orbits are homoclinic related. Moreover, for every $r\in\mathbb{N}_{\geq 2}\cup\{\infty\}$, the set of $C^r$ vector fields exhibiting a geometric Lorenz attractor is an open subset of $\mathscr{X}^r(M^3)$.

\begin{proposition}\label{prop:Lorenz}
	Let $r\in\mathbb{N}_{\geq2}\cup\{\infty\}$  and $X\in\mathscr{X}^r(M^3)$. If $X$ exhibits a geometric Lorenz attractor $\Lambda$ with attracting region $U$, then $\Lambda$ is a singular hyperbolic homoclinic class of $X$, and the orbits of every two periodic points $p,q\in\Lambda$ are homoclinic related.
	
	Moreover, there exists a $C^r$-neighborhood $\cU$ of $X$ in $\mathscr{X}^r(M^3)$, such that for every $Y\in\cU$, $U$ is an attracting region of $Y$, and the maximal invariant set $\Lambda_Y=\bigcap_{t>0}\phi_t^Y(U)$ is a geometric Lorenz attractor.
	
\end{proposition}

\begin{remark}
	The proof of Proposition~\ref{prop:Lorenz} is in Appendix~\ref{Section:robust}.
	The first part  is classical and the ``moreover'' part
	relies on a beautiful observation in \cite{am2}, which shows the stable foliation of the Poincar\'e map in the cross section is $C^{1+}$-smooth.
\end{remark}

\subsection{A $C^r$-connecting theorem of geometric Lorenz attractors}\label{subsec:connecting}

This subsection is devoted to prove Theorem~\ref{Thm:A-connecting} which is a $C^r$-connecting theorem for geometric Lorenz attractors. One can see that Theorem~\ref{Thm:A-connecting} follows from the following technical one with more details.

\begin{theorem}[Theorem~\ref{Thm:A-connecting} revisited]\label{thm:Lorenz-connecting}
	Assume $r\in\mathbb{N}_{\geq2}\cup\{\infty\}$ and $X\in\mathcal{X}^r(M^3)$ admits a geometric Lorenz attractor $\Lambda$ with singularity $\sigma$ and attracting region $U$. Let $\Sigma$ be the cross section of $\Lambda$ from Notation~\ref{notat:section}, and $z\in W^u(\sigma)\cap\Sigma$ be a point satisfying $\phi_t^X(z)\notin\Sigma$ for every $t<0$.
	
	For every $C^r$-neighborhood $\cV$ of $X$, every critical element $p\in\Lambda$ and every $\delta_0>0$, there exist $0<\delta<\delta_0$ and $Y\in\cV$, such that the maximal invariant set $\Lambda_Y=\bigcap_{t>0}\phi_t^Y(U)$ is a geometric Lorenz attractor satisfying
	\begin{enumerate}
		\item The vector field $Y$ and $\delta$-neighborhood $B(z,\delta)$ of $z$ satisfy
		$$
		Y|_{M^3\setminus B(z,\delta)}~\equiv~ X|_{M^3\setminus B(z,\delta)}, 
		\qquad {\it and} \qquad
		B(z,\delta)~\cap~\left(\orb(p,\phi_t^X)\cup\{\sigma\}\right)
		~=~\emptyset.
		$$
		This implies
		$\sigma_Y=\sigma\in\sing(\Lambda_Y)$, and 
		$p\in\per(\Lambda_Y)$ with
		$\orb(p,\phi_t^Y)=\orb(p,\phi_t^X)$.
		
		\item The point $z\in W^u(\sigma,\phi_t^Y)$ and there exists $T>0$, such that
		$\phi^Y_T(z)\in W^s_{\it loc}(p,\phi_t^Y)$. This implies 
		$$
		\orb(z,\phi_t^Y)~\subset~
		W^u(\sigma,\phi_t^Y) \cap W^s(\orb(p),\phi_t^Y).
		$$
	\end{enumerate}
\end{theorem}

Before proving Theorem~\ref{thm:Lorenz-connecting}, we would like to state the following direct corollary by considering the critical element $p$ as the singularity $\sigma$ itself, which shows that $C^r$-densely, the singularity in the geometric Lorenz attractor admits a homoclinic orbit.

\begin{corollary}\label{Cor:dense-loop}
	
	Assume $r\in\mathbb{N}_{\geq2}\cup\{\infty\}$.  There exists a dense subset $\mathcal{D}^r_h\subset \mathscr{X}^r(M^3)$, such that if $X\in \mathcal{D}^r_h$ admits a geometric Lorenz attractor with singularity $\sigma$, then $\sigma$ has a homoclinic orbit.
\end{corollary}

Now we give the proof of Theorem~\ref{thm:Lorenz-connecting}.
\begin{proof}[Proof of Theorem~\ref{thm:Lorenz-connecting}]
	
	By Proposition \ref{prop:Lorenz}, we can shrink the $C^r$-neighborhood $\cV$ of $X$ if necessary, such that for every $Y\in\cV$, the maximal invariant set $\Lambda_Y=\bigcap_{t>0}\phi_t^Y(U)$ is a geometric Lorenz attractor.
	
	There are two points $z^\pm\in W^u(\sigma)\cap\Sigma$ satisfying $\phi_t^X(z^\pm)\notin\Sigma$ for every $t<0$. 
	We assume $z=z^+=(-1,y^+)\in\Sigma$. 
	The proof for $z=z^-$ is the same. 
	
	For the critical element $p\in\Lambda$, there are two possibilities: either $p=\sigma$, or $p$ is a periodic point. 
	\begin{itemize}
		\item  If $p=\sigma$, there exists $\delta\in(0,\delta_0)$, such that  $B(z,\delta)\cap W^s_{\it loc}(\sigma)=\emptyset$. We only need to show there exists a vector field $Y\in\cV$ which satisfies $Y|_{M^3\setminus B(z,\delta)}~\equiv~ X|_{M^3\setminus B(z,\delta)}$, such that for some $T>0$, 
		$$
		z\in W^u(\sigma,\phi_t^Y)\cap\Sigma,
		\qquad {\rm and} \qquad
		\phi_T^Y(z) \in l=
		\Sigma\cap W^s_{\it loc}(\sigma,\phi_t^Y).
		$$	
		\item	If $p$ is a periodic point, then we can assume $p\in\Sigma$ with 
		$$
		p=(x_p,y_p)\in\big((-1,1)\setminus\{0\}\big)\times[-1,1]
		\quad {\rm and} \quad
		l_p=\{x_p\}\times[-1-\epsilon,1+\epsilon]\subset W^s_{\it loc}(\orb(p),\phi_t^X).
		$$
		There exists $\delta\in(0,\delta_0)$, such that
		$B(z,\delta)\cap
		\left(\bigcup_{t\geq0}\phi_t^X(l_p)\cup\{\sigma\}\right)
		=\emptyset$.
		We need to show there exists a vector field $Y\in\cV$ which  satisfies 
		$Y|_{M^3\setminus B(z,\delta)}~\equiv~ X|_{M^3\setminus B(z,\delta)}$, such that for some $T>0$,
		$$
		z\in W^u(\sigma,\phi_t^Y)\cap\Sigma,
		\qquad {\rm and} \qquad
		\phi_T^Y(z)\in l_p
		\subset\Sigma\cap W^s_{\it loc}(\orb(p),\phi_t^Y).
		$$
	\end{itemize}
	From now on, we fix $\delta\in(0,\delta_0)$ which satisfies the above property. Recall that the cross section $\Sigma=[-1-\epsilon,1+\epsilon]^2$ for some $\epsilon>0$ and $P:\Sigma\setminus l\rightarrow\Sigma$ is the Poincar\'e map.

	\begin{claim}\label{clm:nbhd}
		There exist $\eta\in(0,\epsilon)$ and $t_0>0$, such that the cross section
		$\Sigma_z=[-1-\eta,-1+\eta]\times[y^+-\eta,y^++\eta] \subset\Sigma$ at $z$ satisfies
		\begin{enumerate}
			\item\label{item:1} $P(\Sigma_z)\cap\Sigma_z=\emptyset$ and there exists $\eta'>0$ such that $P^{-1}(\Sigma_z)\subset(0,\eta')\times[-1-\epsilon,1+\epsilon]$ and
			$f'(x)>10$ for every $x\in(0,\eta')$.
			\item\label{item:2} The set $\phi_{[0,3t_0]}^X(\Sigma_z)=\bigcup_{t\in[0,3t_0]}\phi_t^X(\Sigma_z)~\subset~ B(z,\delta)$ and $\phi_t^X(z)\notin\phi_{[0,3t_0]}^X(\Sigma_z)$ for each $t<0$. 
			\item\label{item:3} We have $(1-10^{-3})f'(x_1)<f'(x_2)<(1+10^{-3})f'(x_1)$ for every $x_1,x_2\in[-1-\eta,-1+\eta]$.
		\end{enumerate} 
	\end{claim}
	
	\begin{proof}
		The fact that $f(-1)>-1$ implies if $\eta$ is small enough, then $f([-1-\eta,-1+\eta])\cap[-1-\eta,-1+\eta]=\emptyset$. Thus $P(\Sigma_z)\cap\Sigma_z=\emptyset$.
		
		Since $P(\Sigma)\subset[-1,1]\times(-1,1)$, if a point $w\in\Sigma$ satisfies $P(w)\in\Sigma_z$, then $\pi_x\circ P(w)\in(-1,-1+\eta]$. 
		The one-dimensional map $f$ satisfies $f(-1-\epsilon)>-1$, so if $\eta<f(-1-\epsilon)+1$, then there exists some $\eta'>0$, such that $P^{-1}(\Sigma_z)\subset   (0,\eta']\times[-1-\epsilon,1+\epsilon]$. 
		If $\eta\rightarrow 0^+$, then $\eta'\rightarrow0^+$, and $f'(x)\rightarrow+\infty$ as $x\in(0,\eta')$. So we only need to shrink $\eta$ to take $\eta'$ small enough, such that $f'(x)>10$ for every $x\in(0,\eta')$. This shows item~\ref{item:1}.

		To verify item~\ref{item:2}, note that $\Sigma_z=[-1-\eta,-1+\eta]\times[y^+-\eta,y^++\eta]$ is an $\eta$-square neighborhood of $z=(-1,y^+)$ in the cross section $\Sigma$ and there exists $t_0>0$ such that the orbit segment $\phi_{[-t_0,4t_0]}^X(z)\subset B(z,\delta)$. 
		Thus shrinking $\eta$ again if necessary, the flow box $\phi_{[0,3t_0]}^X(\Sigma_z)=\bigcup_{t\in[0,3t_0]}\phi_t^X(\Sigma_z)$ would be contained in the open ball $B(z,\delta)$. Moreover, recall that $z\in W^u(\sigma)$ satisfies $\phi^X_t(z)\cap \Sigma=\emptyset$ for each $t<0$ and $\Sigma_z\subset\Sigma$, thus $\phi_t^X(z)\notin\phi_{[0,3t_0]}^X(\Sigma_z)$ for each $t<0$. This proves item~\ref{item:2}.
		
		Finally, note that $f'$ is continuous since the quotient map $f$ is $C^1$-smooth.
		Thus by shrinking $\eta$ a last time if necessary, the following satisfies 
		\[\frac{1}{\sqrt{1+10^{-3}}}f'(-1)<f'(t)<\sqrt{1+10^{-3}}f'(-1) \text{~~for every~~} t\in[-1-\eta,-1+\eta].\]
		As a consequence, for every $x_1,x_2\in[-1-\eta,-1+\eta]$, one has
		\[1-10^{-3}<\frac{1}{1+10^{-3}}<\frac{f'(x_2)}{f'(x_1)}<1+10^{-3},\] 
		and thus $(1-10^{-3})f'(x_1)<f'(x_2)<(1+10^{-3})f'(x_1)$ which is item~\ref{item:3}.
	\end{proof}
	
	The subset $\phi_{[0,3t_0]}^X(\Sigma_z)$ admits a natural coordinate: for every $w\in\phi_{[0,3t_0]}^X(\Sigma_z)$,
	there exist a unique point $w_0=(x_0,y_0)\in\Sigma_z=[-1-\eta,-1+\eta]\times[y^+-\eta,y^++\eta]$ and $t_w\in[0,3t_0]$, such that $w=\phi_{t_w}^X(w_0)$. Then we denote 
	$$
	w=(x_0,y_0,t_w)~\in~ \phi_{[0,3t_0]}^X(\Sigma_z)
	=[-1-\eta,-1+\eta]\times[y^+-\eta,y^++\eta]\times[0,3t_0].
	$$ 
	We point out here that this coordinate is only $C^1$-smooth.
	
	Let $Z\in\mathscr{X}^\infty(M^3)$ be a smooth vector field satisfying $\supp(Z)\subset\phi_{[t_0,2t_0]}^X(\Sigma_z)\subset B(z,\delta)$ such that if we represent $Z$ in the coordinate $\{(x,y,t)\}$ as
	$$
	Z(x,y,t)=
	a(x,y,t)\cdot\frac{\partial}{\partial x}+
	b(x,y,t)\cdot\frac{\partial}{\partial y}+
	c(x,y,t)\cdot\frac{\partial}{\partial t}
	$$
	then for every $(x,y,t)\in\phi_{[t_0,2t_0]}^X(\Sigma_z)$, it satisfies the following properties
	\begin{enumerate}
		\item $\max\big\{ |b(x,y,t)|,|c(x,y,t)| \big\}\leq  \beta|a(x,y,t)|$ and $|a(x,y,t)|<\beta$ 
		where $0<\beta<10^{-3}$ is a small constant to be determined in Claim~\ref{clm:vector-field-family} below.
		\item $a(x,y,t)\geq0$ and $a(-1,y^+,3t_0/2)>0$.
	\end{enumerate}
	Here since the coordinate $\{(x,y,t)\}$ is $C^1$-smooth, the three functions $a,b,c$ are only uniformly continuous.
	
	\begin{claim}\label{clm:vector-field-family}
		Let $X_s=X+s\cdot Z\in\mathscr{X}^r(M^3)$ for $s\in[0,1]$. Then there exists $\beta>0$ such that the $C^r$-vector field family $\{X_s\}_{s\in[0,1]}$ satisfies the following properties:
		\begin{enumerate}
			\item\label{item:convergence} $X_0=X$ and $X_s$ converges to $X$ in $C^r$-topology as $s\rightarrow 0$.
			\item\label{item:unstable} For every $s\in[0,1]$, $\sigma$ is the singularity of $X_s$ and $z\in W^u(\sigma,\phi_t^{X_s})$.
			\item\label{item:support} For every $s\in[0,1]$, $X_s|_{\phi_{[t_0,2t_0]}^X(\Sigma_z)}\equiv X|_{\phi_{[t_0,2t_0]}^X(\Sigma_z)}$.
			\item\label{item:return map} For every $s\in[0,1]$, the first return map $P_s:\Sigma\setminus l\rightarrow\Sigma$ is well-defined for the flow $\phi_t^{X_s}$, and satisfies 
			$$
			P_s|_{\Sigma\setminus\Sigma_z}~\equiv~
			P|_{\Sigma\setminus\Sigma_z}.
			$$
			\item\label{item:curve} For every $w\in\Sigma_z$, the curve $P_s(w):[0,1]\rightarrow\Sigma$ is $C^r$-smooth with respect to $s$, and  satisfies 
			$$
			\frac{{\rm d}\pi_x\circ P_s(w)}{{\rm d}s}\geq 0,
			\qquad \text { and } \qquad
			\frac{{\rm d}P_s(w)}{{\rm d}s}|_{s=s'}\in
			\cC_{\alpha}(P_{s'}(w)), ~\forall s'\in[0,1].
			$$ 
			Moreover, for the point $z$, the curve $P_s(z):[0,1]\rightarrow\Sigma$ satisfies
			${\rm d}\pi_x\circ P_s(z)/{\rm d}s>0$ at the point $s=0$.
		\end{enumerate} 
	\end{claim}
	
	\begin{proof}[Proof of the claim]
		The definitions of $Z$ and $X_s$ imply the first three items directly. The vector field $X$ on $\phi_{[0,3t_0]}^X(\Sigma_z)$ is equal to $\partial/\partial t$ everywhere. For every $w$ in the boundary of $\Sigma_z$, $Z|_{\phi_{[0,3t_0]}^X(w)}\equiv 0$, this implies $\phi_t^{X_s}(w)=\phi_t^X(w)$ for every $t\in[0,3t_0]$. Moreover, since
		$$
		\max_{(x,y,t)\in\phi_{[0,3t_0]}^X(\Sigma_z)}
		\big\{~|a(x,y,t)|~,~|b(x,y,t)|~,~|c(x,y,t)|~\big\}
		~<~\beta,
		$$
		hence taking $0<\beta<10^{-3}$ small, for every $s\in[0,1]$ and every $w\in\Sigma_z$, there exists some $t=t(w,s)$ which is close to $3t_0$, such that $\phi_t^{X_s}(w)\in\phi_{3t_0}^X(\Sigma_z)$. Thus the Poincar\'e map from $\Sigma_z$ to $\Sigma$ associated to $\phi_t^{X_s}$ is well defined for every $s\in[0,1]$.
		
		For every $w\in\Sigma\setminus\Sigma_z$, its positive orbit does not intersect $\phi_{[0,3t_0]}^X(\Sigma_z)$ before it enters $\Sigma$ again. This implies for every $s\in[0,1]$, the first return map $P_s:\Sigma\setminus l\rightarrow\Sigma$ is well-defined for the flow $\phi_t^{X_s}$, and satisfies 
		$P_s|_{\Sigma\setminus\Sigma_z}~\equiv~
		P|_{\Sigma\setminus\Sigma_z}$. This proves item~\ref{item:return map}.
		
		\vskip 2mm
		
		Let
		$$
		P_{s,1}:\Sigma_z\rightarrow\phi_{3t_0}^X(\Sigma_z),
		\qquad {\rm and} \qquad
		P_{s,2}:\phi_{3t_0}^X(\Sigma_z)\rightarrow\Sigma
		$$
		be Poincar\'e maps of $\phi_t^{X_s}$ from $\Sigma_z$ to $\phi_{3t_0}^X(\Sigma_z)$ and from $\phi_{3t_0}^X(\Sigma_z)$ to $\Sigma$ respectively. Then we have $P_s=P_{s,2}\circ P_{s,1}$. 
		Let the coordinate on $\phi_{3t_0}^X(\Sigma_z)$ be induced  by pushing the coordinate of $\Sigma_z\subset\Sigma$ by $\phi_{3t_0}^X$, then $P_{s,2}\equiv P|_{\Sigma_z}$ in this coordinate.
		
		Since $X_s\in \mathscr{X}^r(M^3)$ is depends continuously with respect to $s$ in $\mathscr{X}^r(M^3)$ and the cross section $\Sigma$ is a $C^r$ smooth surface, the curve $P_s(w):[0,1]\rightarrow\Sigma$ is $C^r$-smooth with respect to $s$ for every $w\in\Sigma_z$. 
		
		The property $a(x,y,t)\geq 0$ everywhere implies for every $w\in\Sigma_z$, ${\rm d}\pi_x\circ P_{s,1}(w)/{\rm d}s\geq0$ everywhere. Here $\pi_x$ is the projection to $x$-coordinate on $\phi_{3t_0}^X(\Sigma_z)$. Since $f'(x)\geq\lambda_0$ everywhere, we have
		$$
		\frac{{\rm d}\pi_x\circ P_s(w)}{{\rm d}s}~=~
		\frac{{\rm d}f}{{\rm d}x}\cdot
		\frac{{\rm d}\pi_x\circ P_{s,1}(w)}{{\rm d}s}
		~\geq~0.
		$$
		Moreover, since $|b(x,y,t)|\leq\beta\cdot|a(x,y,t)|\leq10^{-3}\cdot|a(x,y,t)|$ for every $(x,y,t)$, hence for every $w\in\Sigma_z$ and every $0\leq s_1<s_2\leq 1$, we have
		$$
		\left|\pi_y\circ P_{s_2,1}(w)-\pi_y\circ P_{s_1,1}(w)\right|
		~\leq~ 10^{-3}\cdot
		\left|\pi_x\circ P_{s_2,1}(w)-\pi_x\circ P_{s_1,1}(w)\right|.
		$$
		This implies ${\rm d}P_{s,1}/{\rm d}s$ is tangent to the cone field $D\phi_{3t_0}^X\left(\cC_{\alpha}|_{\Sigma_z}\right)$ everywhere on $\phi_{3t_0}^X(\Sigma_z)$. Since $P_{s,2}\equiv P|_{\Sigma_z}$, and $P$ preserves the cone field $\cC_{\alpha}$(Lemma \ref{lem:cone} and Notation \ref{notat:section}), the curve $P_s(w):[0,1]\rightarrow\Sigma$ is tangent to $\cC_{\alpha}$ everywhere.

		For the point $z=(-1,y^+)\in\Sigma_z$, since $a(-1,y^+,3t_0/2)>0$, there exists a neighborhood of $(-1,y^+,3t_0/2)$, such that for every $(x,y,t)$ in this neighborhood, we have $a(x,y,t)>0$. Thus there exists some $s'>0$ and $\kappa>0$, such that for every $0<s<s'$, we have $\pi_x\circ P_{s,1}(z)\geq-1+\kappa\cdot s$. Thus we have
		$$
		\frac{{\rm d}\pi_x\circ P_s(z)}{{\rm d}s}|_{s=0}~=~
		\frac{{\rm d}f}{{\rm d}x}|_{x=-1}\cdot
		\frac{{\rm d}\pi_x\circ P_{s,1}(w)}{{\rm d}s}|_{s=0}
		~\geq~\lambda_0\cdot\kappa>0.
		$$
		Then item~\ref{item:curve} is proved and the proof of the claim is completed.
	\end{proof}
	
	From Notation \ref{notat:section}, there exists $\lambda_0\in(\sqrt{2},2)$, such that $f'(x)>\lambda_0$ for every $x\in[-1-\epsilon,1+\epsilon]$. We fix the $C^\infty$-vector field $Z$ to satisfy Claim~\ref{clm:vector-field-family} and a constant $\lambda\in(\sqrt{2},\lambda_0)$.
	
	\begin{claim}\label{clm:small-perturb}
		There exists a constant $\tau\in(0,1)$ such that $X_s\in\cV$ for every $s\in[0,\tau]$, and satisfies the following two properties:
		\begin{enumerate}
			\item\label{item:esti-return-map} For every $s\in[0,\tau]$, the Poincar\'e return map $P_s:\Sigma\setminus l\rightarrow\Sigma$ satisfies for every $(x,y)\in\Sigma\setminus l$
			\begin{align*}
				\frac{\partial\pi_x\circ P_s(x,y)}{\partial x}>\lambda,
				\qquad &{\it and} \qquad
				\frac{\partial\pi_x\circ P_s(x,y)}{\partial y}
				<10^{-3}(\lambda-\sqrt{2}); \\
				\left|\frac{\partial\pi_y\circ P_s(x,y)}{\partial x}\right|<1,
				\qquad &{\it and} \qquad
				\left|\frac{\partial\pi_y\circ P_s(x,y)}{\partial y}\right|
				<1.
			\end{align*}
			\item\label{item:esti-curve} Let $\gamma_0:[0,\tau]\rightarrow\Sigma$ be defined as $\gamma_0(s)=P_s(z)$, then we have
			$$
			\frac{{\rm d}\pi_x\circ \gamma_0(s)}{{\rm d}s}|_{s=s'}
			~\geq~
			\frac{1}{2}\cdot
			\frac{{\rm d}\pi_x\circ \gamma_0(s)}{{\rm d}s}|_{s=0},
			\qquad \forall s'\in[0,\tau]
			$$
		\end{enumerate} 
	\end{claim}
	
	\begin{proof}[Proof of the claim]
		Since $X_s\rightarrow X$ in $\mathscr{X}^r(M^3)$ and $P_s|_{\Sigma\setminus\Sigma_z}\equiv P|_{\Sigma\setminus\Sigma_z}$, so we have $P_s\rightarrow P$ in $C^1$-topology under the coordinate $\Sigma=[-1-\epsilon,1+\epsilon]^2$. On the other hand, 
		$$
		\frac{\partial\pi_x\circ P(x,y)}{\partial x}
		\equiv f'(x)\geq\lambda_0,
		\qquad {\rm and} \qquad
		\frac{\partial\pi_x\circ P(x,y)}{\partial y}
		\equiv \frac{f(x)}{y} \equiv 0
		\qquad \forall (x,y)\in\Sigma\setminus l.
		$$
		So there exists $\tau_1>0$, such that for every $s\in[0,\tau]$, $X_s\in\cV$ whose Poincar\'e map $P_s$ satisfies 
		$\partial\pi_x\circ P_s(x,y)/\partial x>\lambda$
		and
		$\partial\pi_x\circ P_s(x,y)/\partial y<10^{-3}(\lambda-\sqrt{2})$
		for every $(x,y)\in\Sigma\setminus l$. The proof for
		$|\partial\pi_y\circ P_s(x,y)/\partial x|<1$ and $|\partial\pi_y\circ P_s(x,y)/\partial y|<1$ is the same.
		
		Since $P_s(z)$ is $C^r$-smooth with respect to $s\in[0,1]$, so the function $\pi_x\circ P_s(z)$ is $C^1$-smooth with respect to $s$. There exists $\tau_2>0$, such that for every $s'\in[0,\tau_2]$, it satisfies
		$$
		\frac{{\rm d}\pi_x\circ \gamma_0(s)}{{\rm d}s}|_{s=s'}
		~\geq~\frac{1}{2}\cdot
		\frac{{\rm d}\pi_x\circ \gamma_0(s)}{{\rm d}s}|_{s=0}
		$$
		We take $\tau=\min\{\tau_1,\tau_2\}$ finishing the proof of the claim.
	\end{proof}
	
	Fix $\tau\in(0,1)$ from Claim~\ref{clm:small-perturb}. For every $n\geq1$, let $D_n\subset[0,\tau]$ be the set consisting of points $s$ satisfying there exists $m<n$, such that $P_s^m\circ\gamma_0(s)\in l$. Then $D_n$ is a finite subset of $[0,\tau]$ with $D_n\subseteq D_{n+1}$.
	
	\begin{claim}\label{clm:1-dim-expanding}
		For every $n\geq0$, the map 
		$$
		\gamma_n(s)=\left(x_n(s),~y_n(x)\right)=P_s^n\circ\gamma_0(s):
		~[0,\tau]\setminus D_n\longrightarrow\Sigma
		$$ 
		is well defined. For each connected component $J$ of $[0,\tau]\setminus D_n$, $\gamma_n$ is a diffeomorphism from $J$ to its image, and satisfies
		$$
		\gamma_n'(s)~\in~\cC_{\alpha}(\gamma_n(s)),
		\qquad {\it and } \qquad
		\frac{{\rm d}\pi_x\circ\gamma_n(s)}{{\rm d}s}~=~x_n'(s)~\geq~\lambda^n\cdot x_0'(s),
		\qquad \forall s\in J.
		$$
	\end{claim}
	
	\begin{proof}[Proof of the claim]
		We proof this claim by induction. When $n=0$, the item~\ref{item:curve} of Claim \ref{clm:vector-field-family} shows that $\gamma_0'(s)\in\cC_{\alpha}(\gamma_0(s))$ for every $s\in[0,\tau]$. And $x_0'(s)\geq\lambda^0 x_0'(s)$ automatically holds for every $s\in[0,\tau]$.
		Now we assume for every $0\leq i\leq n-1$ and every $s\in[0,\tau]\setminus D_i$, 
		$$
		\gamma_i'(s)\in\cC_{\alpha}(\gamma_i(s)).
		\qquad {\rm and } \qquad
		x_i'(s)~\geq~\lambda^i\cdot x_0'(s).
		$$ 
		For every $s\in[0,\tau]\setminus D_n$, we have
		$$
		x_n(s)=\pi_x\circ P_s\left(x_{n-1}(s),y_{n-1}(s)\right),
		\qquad {\rm and} \qquad
		y_n(s)=\pi_y\circ P_s\left(x_{n-1}(s),y_{n-1}(s)\right).
		$$
		
		Now we fix $s_0\in[0,\tau]\setminus D_n$.
		There are two possibilities: either $(x_{n-1}(s_0),y_{n-1}(s_0))\notin\Sigma_z$, or $(x_{n-1}(s_0),y_{n-1}(s_0))\in\Sigma_z$. If $(x_{n-1}(s_0),y_{n-1}(s_0))\notin\Sigma_z$, then there exists a neighborhood of $(x_{n-1}(s_0),y_{n-1}(s_0))$, such that $P_s\equiv P$ in this neighborhood. Lemma \ref{lem:cone} shows that 
		$$
		\gamma_n'(s_0)~=~DP(\gamma_{n-1}'(s_0))
		~\in~\cC_{\alpha}(\gamma_n(s_0)).
		$$
		
		If $(x_{n-1}(s_0),y_{n-1}(s_0))\in\Sigma_z$, item~\ref{item:1} of Claim \ref{clm:nbhd} implies $(x_{n-2}(s_0),y_{n-2}(s_0))\notin\Sigma_z$, $0<x_{n-2}(s_0)<\eta'$ and $f'(x_{n-2}(s_0))>10$. If we denote
		\begin{align*}
			\gamma_{n-2}'(s_0)~&=~
			x_{n-2}'(s_0)\cdot\frac{\partial}{\partial x}+
			y_{n-2}'(s_0)\cdot\frac{\partial}{\partial y}~\in~
			\cC_{\alpha}(\gamma_{n-2}(s_0)). \\
			\gamma_{n-1}'(s_0)~&=~
			x_{n-1}'(s_0)\cdot\frac{\partial}{\partial x}+
			y_{n-1}'(s_0)\cdot\frac{\partial}{\partial y}~\in~
			\cC_{\alpha}(\gamma_{n-1}(s_0)).
		\end{align*}
		Then $\gamma_{n-1}'(s_0)=
		\left(DP|_{(x_{n-2}(s_0),y_{n-2}(s_0))}\right) \big(\gamma_{n-2}'(s_0)\big)$. Thus we have
		\begin{align*}
			\left|y_{n-2}'(s_0)\right|\leq
			\alpha\cdot\left|x_{n-2}'(s_0)\right|,
			\qquad
			|x_{n-1}'(s_0)|=f'(x_{n-2}(s_0))\cdot|x_{n-2}'(s_0)|
			\geq 10|x_{n-2}'(s_0)|,
		\end{align*}
		and
		\begin{align*}
			|y_{n-1}'(s_0)|~&\leq~ 
			\frac{\partial H}{\partial x}|_{\gamma_{n-2}(s_0)}
			\cdot|x_{n-2}'(s_0)|~+~
			\frac{\partial H}{\partial y}|_{\gamma_{n-2}(s_0)}
			\cdot|y_{n-2}'(s_0)| \\
			&\leq~ |x_{n-2}'(s_0)|~+~|y_{n-2}'(s_0)| \\
			&\leq~(1+\alpha)|x_{n-2}'(s_0)|.
		\end{align*}
		
		Now we calculate $\gamma_n'(s_0)$. Since 
		$x_n(s)=\pi_x\circ P_s\left(x_{n-1}(s),y_{n-1}(s)\right)$
		and
		$y_n(s)=\pi_y\circ P_s\left(x_{n-1}(s),y_{n-1}(s)\right)$, 
		if we denote $w_{n-1}=\left(x_{n-1}(s_0),y_{n-1}(s_0)\right)$, then 
		\begin{align*}
			x_n'(s_0)=
			\frac{\partial \pi_x\circ P_{s_0}}{\partial x}
			\cdot x_{n-1}'(s_0) +
			\frac{\partial \pi_x\circ P_{s_0}}{\partial y}
			\cdot y_{n-1}'(s_0) +
			\frac{\partial \pi_x\circ P_s(w_{n-1})}{\partial s}
			|_{s=s_0}.
		\end{align*}     
		From item~\ref{item:esti-return-map} of Claim \ref{clm:small-perturb}, we have
		\begin{align*}
			\frac{\partial \pi_x\circ P_{s_0}}{\partial x}
			\cdot x_{n-1}'(s_0)
			\geq \lambda\cdot x_{n-1}'(s_0) 
			\quad {\rm and} \quad   
			\left|\frac{\partial \pi_x\circ P_{s_0}}{\partial y}
			\cdot y_{n-1}'(s_0)\right|
			\leq 10^{-3}(\lambda-\sqrt{2})\cdot|y_{n-1}'(s_0)|.
		\end{align*}     
		Since $|y_{n-1}'(s_0)|\leq\alpha\cdot x_{n-1}'(s_0)$ and $10^{-3}\alpha<1$, we have
		\begin{align}\label{equ:x-expanding}
			x_n'(s_0)\geq 
			\sqrt{2}\cdot x_{n-1}'(s_0) + \frac{\partial \pi_x\circ P_s(w_{n-1})}{\partial s}
			|_{s=s_0}
		\end{align}
		
		On the other hand, we have
		$$
		y_n'(s_0)=\frac{\partial \pi_y\circ P_{s_0}}{\partial x}
		\cdot x_{n-1}'(s_0) +
		\frac{\partial \pi_y\circ P_{s_0}}{\partial y}
		\cdot y_{n-1}'(s_0) +
		\frac{\partial \pi_y\circ P_s(w_{n-1})}{\partial s}
		|_{s=s_0}.
		$$
		Again, item~\ref{item:esti-return-map} of Claim \ref{clm:small-perturb} shows that $|\partial \pi_y\circ P_{s_0}/\partial x|<1$ and $|\partial \pi_y\circ P_{s_0}/\partial y|<1$. Thus we have
		\begin{align}\label{equ:y-small}
			|y_n'(s)|<(1+\alpha)\cdot x_{n-1}'(s_0)+
			\left|\frac{\partial \pi_y\circ P_s(w_{n-1})}{\partial s}
			|_{s=s_0}\right|.
		\end{align}
		
		Item~\ref{item:curve} of Claim \ref{clm:vector-field-family} shows
		that 
		$$
		\frac{{\rm d}P_s(w_{n-1})}{{\rm d}s}|_{s=s_0}\in
		\cC_{\alpha}(P_{s_0}(w))
		\qquad {\rm and} \qquad
		\frac{\partial \pi_x\circ P_s(w_{n-1})}{\partial s}|_{s=s_0}\geq 0,
		$$
		Which implies
		$$
		\left|\frac{\partial \pi_y\circ P_s(w_{n-1})}{\partial s}|_{s=s_0}\right|
		~\leq~\alpha\cdot
		\frac{\partial \pi_x\circ P_s(w_{n-1})}{\partial s}|_{s=s_0}.
		$$
		Since $\alpha=1/(\sqrt{2}-1)$ which is equivalent to $(1+\alpha)/\sqrt{2}=\alpha$, Equation \ref{equ:x-expanding} and Equation \ref{equ:y-small} imply
		$$
		|y_n'(s_0)|~\leq~ \alpha\cdot x_n'(s_0),
		\qquad {\it i.e.} \qquad
		\gamma_n'(s_0)\in\cC_{\alpha}(\gamma_n(s_0)).
		$$
		
		Finally, Equation \ref{equ:x-expanding} shows that
		$$
		x_n'(s_0)~\geq~\sqrt{2}\cdot x_{n-1}'(s_0)
		~\geq~ 10\sqrt{2}\cdot x_{n-2}'(s_0)
		~>~ \lambda^2\cdot\lambda^{n-2}
		~=~ \lambda^n.
		$$
		This finishes the proof of the claim.    	
	\end{proof}
	
	We can finish the proof of this theorem by using the idea of Lemma \ref{lem:evenually-onto}. 
	
	\begin{claim}\label{clm:cover}
		There exists some integer $N>0$, such that 
		$$
		(-1,1)~\subseteq~
		\bigcup_{i=0}^N\pi_x\circ\gamma_n\big([0,\tau]\big).
		$$
	\end{claim}
	
	\begin{proof}[Proof of the claim]
		If $(a,b)\subseteq[0,\tau]$ is a connected component of $[0,\tau]\setminus D_n$, then 
		$$
		\pi_x\circ\gamma_n(s)=\pi_x\circ P_s^n\circ\gamma_0(s):
		~[0,\tau]\rightarrow[-1,1]
		$$ 
		is a diffeomorphism from $(a,b)$ to its image. 
		
		Similar to Claim \ref{clm:counting}, we show that
		for every $n\geq1$, let $I=(a,b)$ be a connected component of 
		$[0,\tau]\setminus D_n$, then either $\#\big(I\cap D_{n+2}\big)\leq1$ or 
		$(-1,1)\subseteq \pi_x\circ\gamma_{n+1}(I)\cup \pi_x\circ\gamma_{n+2}(I)$. 
		
		Actually, let $(c,d)=\pi_x\circ\gamma_n(I)$. If $0\notin(c,d)$, then $I\cap D_{n+1}=\emptyset$ and 
		$$
		\pi_x\circ\gamma_{n+1}(I)=
		\pi_x\circ P_s^{n+1}\circ\gamma_0(I)
		$$ 
		is an open interval in $(-1,1)$. Thus there exists at most one point $s_0\in I$ satisfying $\pi_x\circ\gamma_{n+1}(s_0)=0$. This implies $\#\big(I\cap D_{n+2}\big)\leq1$.
		
		If $0\in(c,d)$, then we denote $s_0\in I$ satisfying $\pi_x\circ\gamma_n(s_0)=\pi_x\circ P_{s_0}^n\circ\gamma_0(s_0)=0$. Then two open intervals $(a,s_0)$ and $(s_0,b)$ are two connected components of $I\setminus D_{n+1}$. Moreover, item~\ref{item:return map} of Claim \ref{clm:vector-field-family} shows that 
		$P_s|_{\Sigma\setminus\Sigma_z}\equiv
		P|_{\Sigma\setminus\Sigma_z}$. This implies
		$$
		\lim_{s\rightarrow s_0^-}\pi_x\circ\gamma_{n+1}(s)=1,
		\qquad {\rm and} \qquad
		\lim_{s\rightarrow s_0^+}\pi_x\circ\gamma_{n+1}(s)=-1.	
		$$
		Thus there exists $e,f\in(-1,1)$, such that
		$$
		\pi_x\circ\gamma_{n+1}\big( (a,s_0) \big)=(e,1),
		\qquad {\rm and} \qquad 
		\pi_x\circ\gamma_{n+1}\big( (s_0,b) \big)=(-1,f).
		$$
		If both of them don't contain $0$, then we have $\#\big(I\cap D_{n+2}\big)=\#\big(I\cap D_{n+1}\big)=1$.
		
		Otherwise, we have either  $0\in(e,1)$ or $0\in(-1,f)$. Thus either 
		$$
		[0,1)\subset \pi_x\circ\gamma_{n+1}\big( (a,s_0) \big)
		\subset \pi_x\circ\gamma_{n+1}(I),
		\qquad {\rm or} \qquad
		(-1,0]\subset \pi_x\circ\gamma_{n+1}\big( (s_0,b) \big)
		\subset \pi_x\circ\gamma_{n+1}(I).
		$$
		In both cases, we can apply $P_s|_{\Sigma\setminus\Sigma_z}\equiv
		P|_{\Sigma\setminus\Sigma_z}$ and show that 
		$$
		(-1,1)\subseteq \pi_x\circ\gamma_{n+1}(I)\cup \pi_x\circ\gamma_{n+2}(I).
		$$
		
		Now we claim there exists some $n>0$ and a connected component $I$ of $[0,\tau]\setminus D_n$, such that $(-1,1)\subseteq \pi_x\circ\gamma_{n+1}(I)\cup \pi_x\circ\gamma_{n+2}(I)$. Otherwise, we have
		$$
		\# D_2\leq1, \qquad 
		\# D_4\leq\# D_2+\# D_2+1\leq 3,
		\qquad \cdots \qquad 
		\# D_{2n}\leq \# D_{2n-2}+\# D_{2n-2}+1\leq 2^n-1
		$$
		This implies $[0,\tau]\setminus D_{2n}$ has at most $2^n$ connected components. So there exists a connected component $I_{2n}\subset [0,\tau]\setminus D_{2n}$, such that its length satisfies $|I_{2n}|\geq 2^{-n}\tau$.
		
		Since $\pi_x\circ\gamma_{2n}$ is a diffeomorphism on each connected component of $[0,\tau]\setminus D_{2n}$, and Claim \ref{clm:1-dim-expanding} shows that
		$$
		\frac{{\rm d}\pi_x\circ\gamma_n(s)}{{\rm d}s}~=~x_n'(s)~\geq~\lambda^n\cdot x_0'(s),
		\qquad \forall s\in I_{2n},
		$$
		the length of the interval $\pi_x\circ\gamma_n(I_{2n})$ satisfies
		$|\pi_x\circ\gamma_n(I_{2n})| 
		\geq \lambda^{2n}\cdot|\gamma_0(I_{2n})|$.
		
		Finally,  item~\ref{item:esti-curve} of Claim \ref{clm:small-perturb} shows that 
		$$
		|\gamma_0(I_{2n})|~\geq~\frac{1}{2}\cdot
		\frac{{\rm d}\pi_x\circ \gamma_0(s)}{{\rm d}s}|_{s=0}
		\cdot|I_{2n}|
		~\geq\frac{1}{2}\cdot
		\frac{{\rm d}\pi_x\circ \gamma_0(s)}{{\rm d}s}|_{s=0}\cdot 2^{-n}\tau.
		$$
		Thus we have
		$$
		|\pi_x\circ\gamma_n(I_{2n})|~\geq~
		\frac{\tau}{2}\cdot
		\frac{{\rm d}\pi_x\circ \gamma_0(s)}{{\rm d}s}|_{s=0}
		\cdot\left(\frac{\lambda^2}{2}\right)^n
		~\longrightarrow~+\infty,
		\qquad {\rm as} \quad n\rightarrow+\infty.
		$$
		This is absurd because $\pi_x\circ\gamma_n(I_{2n})\subseteq(-1,1)$. So there exists $n>0$ and a connected component $I$ of $[0,\tau]\setminus D_n$, such that $(-1,1)\subseteq \pi_x\circ\gamma_{n+1}(I)\cup \pi_x\circ\gamma_{n+2}(I)$. We take $N=n+2$ to complete the proof of this claim.
	\end{proof}
	
	For the critical point $p\in\Lambda$, recall that $l_p=\{x_p\}\times[-1-\epsilon,1+\epsilon]\subset\Sigma$  with $x_p\in(-1,1)$ (if $p=\sigma$, then $x_p=0$). Claim \ref{clm:cover} shows there exist $s\in[0,\tau]$ and $n>0$, such that 
	$$
	x_p~=~\pi_x\circ P_s^n\circ\gamma_0(s)
	~=~\pi_x\circ P_s^n\circ P_s(z).
	$$
	This implies for the vector field $Y=X_s\in\cV$, there exists $T>0$, such that $\phi_T^Y(z)\in l_p$.
	
	Since $X_s|_{M^3\setminus B(z,\delta)}\equiv X|_{M^3\setminus B(z,\delta)}$ and 
	$B(z,\delta)\cap
	\left(\bigcup_{t\geq0}\phi_t^X(l_p)\cup\{\sigma\}\right)
	=\emptyset$, we have
	$l_p\subset W^s_{\it loc}(\orb(p),\phi_t^Y)$. 
	The fact $z\in W^u(\sigma,\phi_t^Y)=W^u(\sigma,\phi_t^{X_s})$ implies 
	$$
	\orb(z,\phi_t^Y)~\subset~
	W^u(\sigma,\phi_t^Y) \cap W^s(\orb(p),\phi_t^Y).
	$$
	This completes the proof of this theorem.	
\end{proof}

\section{The space of ergodic measures for Lorenz attractors}\label{Section:measure}
In this section, we aim to prove Theorem~\ref{Thm:B} and Corollary~\ref{Cor:Cr-entropy-support}. 
We would also obtain properties concerning path connectedness and denseness of periodic measures for ergodic measure space for more general invariant sets.

\subsection{Path connectedness}

This subsection devotes to give a general criterion for path connectedness (or arcwise connectedness) of  the ergodic measure space. More precisely, for a  homoclinic class $\Lambda$ of a vector field $X\in\mathscr{X}^1(M)$, if all periodic orbits  are homoclinically related with each other, then the set of ergodic measures supported on $\Lambda$ which can be approximated by periodic measures is path connected.  A related conclusion for diffeomorphisms is proved in~\cite{Go-Pe}.

\begin{proposition}\label{Prop:connectedness}
	Let $X\in\mathscr{X}^1(M)$ and $\Lambda$ be a  homoclinic class of $X$ such that all periodic points in $\Lambda$  are hyperbolic and any two periodic points in $\Lambda$ are homoclinically related. 
	Then the set $\overline{\mathcal{M}_{per}(\Lambda)}\cap\mathcal{M}_{erg}(\Lambda)$ is path connected.
\end{proposition}

The proof of Proposition~\ref{Prop:connectedness} follows the idea of ~\cite[Theorem 1.1, 1.2]{Go-Pe}, before which we need some preparations. 
The first one states that any ergodic measure supported on a singular hyperbolic homoclinic class can be accumulated by periodic measures if it is not supported on singularities.

\begin{proposition}\label{Prop:dense periodic measure}
	Let $X\in\mathscr{X}^1(M)$ and $\Lambda$ be a singular hyperbolic homoclinic class. Assume $\mu\in\mathcal{M}_{erg}(\Lambda)$ is not an atomic measure supported on any singularity. Then $\mu\in\overline{\mathcal{M}_{per}(\Lambda)}$.
\end{proposition} 

\begin{remark}
	The proof of Proposition \ref{Prop:dense periodic measure} originates from~\cite[Proposition 1.4]{C-PH} following Katok's arguments~\cite{katok} for the case of diffeomorphisms while for vector fields, the arguments essentially follows~\cite[Theorem 5.6]{sgw} by applying Liao's shadowing lemma~\cite{liao-shadowing1,liao-shadowing2}. Thus we omit the proof.
\end{remark}

	The following two lemmas  are the version of~\cite[Lemma 3.1, 3.2]{Go-Pe} for vector fields. Given a periodic point $p$, we denote by $\delta_{\orb(p)}$ the periodic measures corresponding to $\orb(p)$. The first lemma states that any two periodic measures can be connected through a path inside $\overline{\mathcal{M}_{per}(\Lambda)}\cap\mathcal{M}_{erg}(\Lambda)$.
	
	\begin{lemma}\label{Lem:connect per-measure}
		Under hypothesis of Proposition~\ref{Prop:connectedness}, for any two periodic points $p,q\in\Lambda$, there exits a continuous path $\{\nu_t\}_{t\in[0,1]}\subset \overline{\mathcal{M}_{per}(\Lambda)}\cap\mathcal{M}_{erg}(\Lambda)$ such that $\nu_0=\delta_{\orb(p)}$ and $\nu_1=\delta_{\orb(q)}$.
	\end{lemma}
	\begin{proof}
		Since the two hyperbolic periodic orbits $\orb(p)$ and $\orb(q)$ are homoclinically related,  there exists a hyperbolic horseshoe $K\subset \Lambda$ of $X$ containing $p$ and $q$. Note here that the hypberolic horseshoe $K$ of $X$ is the suspension of a hyperbolic horseshoe $\tilde{K}$ for a diffeomorphism (the time one map of the flow) with a continuous roof function. By~\cite[Theorem B]{Sigmund}, the space  $\mathcal{M}_{erg}(\tilde{K})$ is path connected. On the other hand, there is a one-to-one correspondence between $\mathcal{M}_{erg}(\tilde{K})$ and $\mathcal{M}_{erg}(K)$, see for instance~\cite[Chapter 6]{parry-pollicott}. Hence $\mathcal{M}_{erg}(K)$ is path connected. Moreover, $\mathcal{M}_{per}(K)$ is dense in $\mathcal{M}_{erg}(K)$ since $K$ is a hyperbolic horseshoe. Thus $\delta_{\orb(p)}$ and $\delta_{\orb(q)}$ can be connected by a path contained in $\mathcal{M}_{erg}(K)\subset\overline{\mathcal{M}_{per}(\Lambda)}\cap\mathcal{M}_{erg}(\Lambda)$.
	\end{proof}
	
	Next lemma states that if two periodic measures are both close to a same ergodic measure $\mu$, then there is a short path in the ergodic measure space connecting them in the sense that the whole path is close to $\mu$.
	\begin{lemma}\label{Lem:short path}
		Under hypothesis of Proposition~\ref{Prop:connectedness}, for any $\mu\in\mathcal{M}_{erg}(\Lambda)$ and $\varepsilon>0$, there exists $\eta>0$ such that if two periodic points $p,q\in\Lambda$ satisfies $d(\mu,\delta_{\orb(p)})<\eta$ and $d(\mu,\delta_{\orb(q)})<\eta$, then there exists a path $\{\nu_t\}_{t\in[0,1]}\subset \overline{\mathcal{M}_{per}(\Lambda)}\cap\mathcal{M}_{erg}(\Lambda)$ such that $\nu_0=\delta_{\orb(p)}$, $\nu_1=\delta_{\orb(q)}$ and $d(\mu,\nu_t)<\varepsilon$ for any $t\in[0,1]$.
	\end{lemma}
	\begin{proof}
		Since $\orb(p)$ and $\orb(q)$ are homoclinically related, there exist  hyperbolic horseshoes such that every point in which spends arbitrarily large fractions of its orbit shadowing $\orb(p)$ and $\orb(q)$ as closely as we want.
		Thus we can choose a hyperbolic horseshoe $K$ containing $p$ and $q$ such that any ergodic measure supported on $K$ is $\eta$-close to $\delta_{\orb(p)}$ or $\delta_{\orb(q)}$. As a consequence, the horseshoe $K$ is $2\eta$-close to $\mu$ in the measure sense: any ergodic measure supported on $K$ is contained in the $2\eta$-neighborhood of $\mu$. By~\cite[Theorem B]{Sigmund}, there is a continuous path $\{\nu_t\}_{t\in[0,1]}\subset\mathcal{M}_{erg}(K)\subset \overline{\mathcal{M}_{per}(\Lambda)}\cap\mathcal{M}_{erg}(\Lambda)$ connecting $\delta_{\orb(p)}$ and $\delta_{\orb(q)}$, satisfying $d(\mu,\nu_t)<2\eta$ for any $t\in[0,1]$. 
		Then Lemma~\ref{Lem:short path} concludes if we take $\eta=\varepsilon/2$.
	\end{proof}

Now we could complete the proof of Proposition~\ref{Prop:connectedness}.

\begin{proof}[Proof of Proposition~\ref{Prop:connectedness}]	
	Take two ergodic measures $\mu,\nu\in\overline{\mathcal{M}_{per}(\Lambda)}\cap\mathcal{M}_{erg}(\Lambda)$. We only need to show there exists a path $\{\mu_t\}_{t\in[0,1]}\subset \overline{\mathcal{M}_{per}(\Lambda)}\cap\mathcal{M}_{erg}(\Lambda)$ such that $\mu_0=\mu$ and $\mu_1=\nu$. We follow the idea of~\cite[Theorem 1.1]{Go-Pe}. 
	By assumption, there exist two sequences of periodic points $\{p_n\}$ and $\{q_n\}$ in $\Lambda$ such that $\delta_{\orb(p_n)}\rightarrow\mu$ and  $\delta_{\orb(q_n)}\rightarrow\nu$ in the weak*-topology.
	Taking subsequences if necessary, we assume 
	\[d(\delta_{\orb(p_n)},\mu)<\frac{1}{2\cdot3^n} \text{~~and~~} d(\delta_{\orb(q_n)},\nu)<\frac{1}{2\cdot3^n}, \text{~~for all $n\geq 1$.} \]
	By Lemma~\ref{Lem:connect per-measure}, there exits a continuous path $\{\nu_t\}_{t\in[\frac{1}{3},\frac{2}{3}]}\subset \overline{\mathcal{M}_{per}(\Lambda)}\cap\mathcal{M}_{erg}(\Lambda)$ such that $\nu_{\frac{1}{3}}=\delta_{\orb(p_1)}$ and $\nu_{\frac{2}{3}}=\delta_{\orb(q_1)}$. 
	By Lemma~\ref{Lem:short path}, for any $n\geq 1$ there exist two paths 
	$$\{\nu_t\}_{t\in[\frac{1}{3^n},\frac{1}{3^{n+1}}]} \text{ and } \{\nu_t\}_{t\in[1-\frac{1}{3^n},1-\frac{1}{3^{n+1}}]} $$
	in $\overline{\mathcal{M}_{per}(\Lambda)}\cap\mathcal{M}_{erg}(\Lambda)$  connecting $\delta_{\orb(p_n)},\delta_{\orb(p_{n+1})}$ and $\delta_{\orb(q_n)},\delta_{\orb(q_{n+1})}$ respectively and the length of the paths decreases to 0 exponentially when $n$ goes to $\infty$. Therefore  $\{\nu_t\}_{t\in[0,1]}$ is a continuous path in $\overline{\mathcal{M}_{per}(\Lambda)}\cap\mathcal{M}_{erg}(\Lambda)$ satisfying that $\nu_0=\mu$ and $\nu_1=\nu$. 
	This completes the proof.
\end{proof}

\subsection{Homoclinic loop and periodic measures}
In this subsection, we prove that a homoclinic loop of a singularity can be accumulated by periodic orbits with the associated periodic measures accumulating to the singular measure by $C^r$-perturbations. 

\begin{proposition}\label{Prop:loop}
	Let $r\in\mathbb{N}\cup\{\infty\}$ and $X\in\mathscr{X}^r(M)$ be a $C^r$-vector field on $M$. If $\sigma\in\sing(X)$ is a hyperbolic singularity of $X$, and 
	$$
	\Gamma=\orb(z,\phi_t^X)\subset W^s(\sigma,\phi_t^X)\cap W^u(\sigma,\phi_t^X)
	$$
	is a homoclinic orbit of $\sigma$, then there exist a sequence of vector fields $\{X_n\}$ and a sequence of periodic points $p_n\in\per(X_n)$, such that
	\begin{enumerate}
		\item The vector field $X_n$ converges to $X$ in $\mathscr{X}^r(M)$.
		\item $\sigma\in\sing(X_n)$ and the periodic measure 
		$\delta_{\orb(p_n,\phi_t^{X_n})}$ converges to $\delta_{\sigma}$.
		\item The orbit $\orb(p_n,\phi_t^{X_n})$ converges to $\Gamma\cup\{\sigma\}$ in the Hausdorff topology.
	\end{enumerate}
\end{proposition}

\begin{proof}
	Let $d=\dim M$ and $s=\dim E^s(\sigma)$ be the index of $\sigma$, then $\dim E^u(\sigma)=d-s$. Since $\Gamma=\orb(z,\phi_t^X)$ is a homoclinic loop of $X$, we have
	$$
	\lim_{t\rightarrow\pm\infty}\phi_t^X(z)=\sigma.
	$$
	By Grobman-Hartman Theorem \cite[4.10 Theorem]{deM-Palis-book}, there exists a small neighborhood $U(\sigma)$ of $\sigma$, such that the flow $\phi_t^X$ is orbit equivalent to the linear system generated by $DX(\sigma)$. Moreover, shrinking $U(\sigma)$ if necessary, we can assume
	there exist $T_1<0<T_2$ satisfying
	$$
	U(\sigma)\cap\Gamma~=~
	\phi_{(-\infty,T_1)}^X(z)~\cup~\phi_{(T_2,+\infty)}^X(z),
	$$
	where $\phi_{T_1}^X(z)\in W^u_{\it loc}(\sigma,\phi_t^X)$ and $\phi_{T_2}^X(z)\in W^s_{\it loc}(\sigma,\phi_t^X)$.
	
	There exists a $(d-1)$-dimensional $C^\infty$-smooth cross section $\Sigma_1$ which is transverse to $X$ everywhere and satisfies the following properties:
	\begin{itemize}
		\item $\phi_{T_1-1}^X(z)\in\Sigma_1$, and $\phi_{[0,1]}^X(\Sigma_1)\cap\Gamma=\phi_{[T_1-1,T_1]}^X(z)$.
		\item There exists a $s$-dimensional $C^\infty$-smooth disk $\Sigma_1^s\subset\Sigma_1$ such that $\Sigma_1^s$ is transverse to $W^u_{\it loc}(\sigma)$ at $\phi_{T_1-1}^X(z)$.
	\end{itemize}
	Symmetrically, There exists a $(d-1)$-dimensional $C^\infty$-smooth cross section $\Sigma_2$ which is transverse to $X$ everywhere and satisfies the following properties:
	\begin{itemize}
		\item $\phi_{T_2+1}^X(z)\in\Sigma_2$, and  $\phi_{[-1,0]}^X(\Sigma_2)\cap\Gamma=\phi_{[T_2,T_2+1]}^X(z)$;
		\item there exists a $(d-s)$-dimensional $C^\infty$-smooth disk $\Sigma_2^u\subset\Sigma_2$ such that $\Sigma_2^u$ is transverse to $W^s_{\it loc}(\sigma)$ at $\phi_{T_2+1}^X(z)$.
	\end{itemize}
	
	By the $\lambda$-lemma (Inclination Lemma), see \cite[7.2 Lemma]{deM-Palis-book}, there exist a sequence of points $x_n\in\Sigma_2^u\subseteq\Sigma_2$ and $S_n\rightarrow+\infty$, such that
	\begin{enumerate}
		\item The point $x_n\rightarrow\phi_{T_2+1}^X(z)$ in $\Sigma_2^u$ as $n\rightarrow\infty$.
		\item The point $y_n=\phi_{S_n}^X(z)\in\Sigma_1^s\subseteq\Sigma_1$, and $y_n\rightarrow\phi_{T_1-1}^X(z)$ in $\Sigma_1^s$ as $n\rightarrow\infty$.
		\item The orbit segment $\phi_{[0,S_n]}^X(x_n)\subset U(\sigma)$ for every $n$.
	\end{enumerate}
	
	Now we consider two flow box $V_1=\phi_{[T_1-1,T_1]}^X(\Sigma_1)$ and $V_2=\phi_{[T_2,T_2+1]}^X(\Sigma_2)$. Since $x_n\rightarrow\phi_{T_2+1}^X(z)$ in $\Sigma_2^u$ and $y_n\rightarrow\phi_{T_1-1}^X(z)$ in $\Sigma_1^s$, there exists a sequence of vector fields $X_n\in\mathscr{X}^r(M)$, such that
	\begin{enumerate}
		\item The vector field $X_n$ converges to $X$ in $C^r$-topology.
		\item $X_n|_{M\setminus(V_1\cup V_2)}\equiv X|_{M\setminus(V_1\cup V_2)}$.
		\item There exist $a_n,b_n\rightarrow 1$ as $n\rightarrow\infty$, such that
		$$
		\phi_{a_n}^{X_n}(y_n)=\phi_{T_1}^X(z),
		\qquad {\rm and} \qquad
		\phi_{b^n}^{X_n}\big( \phi_{T_2}^X(z) \big)=x_n.
		$$ 
	\end{enumerate}
	This implies the vector field $X_n$ satisfies 
	$\phi_{b_n+S_n+a_n}^{X_n}\big(\phi_{T_2}^X(z)\big)=\phi_{T_1}^X(z)$. 
	So $\Gamma_n=\orb(z,\phi_t^{X_n})$ is a periodic orbit of $X_n$ with period $\pi_n=a_n+b_n+(T_2-T_1)+S_n$. Moreover, we have
	$$
	\phi_{t+b_n}^{X_n}\big( \phi_{T_2}^X(z) \big)=
	\phi_t^{X_n}(x_n)=\phi_t^X(x_n), ~~\forall t\in[0,S_n],
	\quad {\rm and} \quad
	\phi_t^{X_n}(z)=\phi_t^X(z), ~~\forall t\in[T_1,T_2].
	$$

	Recall that $x_n\rightarrow\phi_{T_2+1}^X(z)\in W^s_{\it loc}(\sigma)$ and $y_n\rightarrow\phi_{T_1-1}^X(z)$. Since $(\pi_n-S_n)$ is uniformly bounded above, the fact that
	$$
	\phi_{[0,S_n]}^{X_n}(x_n)=\phi_{[0,S_n]}^X(x_n)\subseteq U(\sigma)
	$$
	and $\phi_t^X$ is orbit equivalent to the linear system generated by $DX(\sigma)$ in $U(\sigma)$ implies $\Gamma_n$ converges to $\Gamma\cup\{\sigma\}$ in the Hausdorff topology and the periodic measure 
	$\delta_{\Gamma_n}$ converges to $\delta_{\sigma}$.
\end{proof}

\subsection{Isolated singular measures: a criterion}

In this subsection, we show that if the unstable manifold of a codimension one singularity is contained in the stable manifold of a periodic orbit, then the Dirac measure of this singularity is isolated in the space of ergodic measures. The idea of this proposition is inspired by the work of J. Palis \cite{palis}, see also \cite{deM-Palis}.

\begin{proposition}\label{Prop:kick-out singularity}
	Let $X\in\mathscr{X}^1(M)$ and $\sigma\in\sing(X)$ be a hyperbolic singularity satisfying $\dim E^u(\sigma)=1$. Assume there exists a hyperbolic periodic point $p\in\Lambda$ such that $W^u(\sigma)\setminus\{\sigma\}\subset W^s(\orb(p))$. Then the Dirac measure $\delta_{\sigma}$ is isolated in $\mathcal{M}_{erg}(X)$.
\end{proposition}

\begin{remark}\label{Rem:isolated-measure}
	Note that $\dim E^u(\sigma)=1$ implies $W^u(\sigma)$ consists two branches, each of which is a single orbit.
	Proposition~\ref{Prop:kick-out singularity} gives a general criterion to obtain isolated ergodic measures inside a (non-trivial) transitive set. For instance, when the hyperbolic singularity $\sigma$ of $X$ is contained in a homoclinic class $\Lambda$ and if one could make $W^u(\sigma)\setminus\{\sigma\}\subset W^s(\orb(p))$ through $C^r (r\geq 1)$ perturbations  where $\orb(p)$ is a hyperbolic periodic orbit, then $\delta_\sigma$ is isolated in $\mathcal{M}_{erg}(Y)$ where $Y$ is the resulting vector field after perturbation. See the dense part of Theorem~\ref{Thm:SH-attractor} in Appendix~\ref{Section:C1}.
\end{remark}

\begin{proof}
	Assume there exists a sequence of ergodic measures $\mu_n$ approximating the Dirac measure $\delta_{\sigma}$. We only need to find a continuous function $\varphi:M\rightarrow\RR$, such that $\int\varphi{\rm d}\mu_n$ does not converge to $\int\varphi{\rm d}\delta_{\sigma}$, then we get a contradiction. 
	
	Since $\sigma$ is a hyperbolic singularity,
	each $\mu_n$ does not support on a singularity. 
	Let $x_n\in\supp(\mu_n)$ be a generic point of $\mu_n$, i.e. the empirical measure equidistributed on the orbit segment $\phi_{[0,T]}(x_n)$ converges to $\mu_n$ as $T\rightarrow+\infty$. We must have $\ind(p)<d-1$. Otherwise, $\orb(p)$ is a periodic sink. Hence as the positive orbit of $x_n$ approaches $\sigma$, it also approaches $W^u(\sigma)$ and enters the attracting region of $\orb(p)$. This contradicts to $x_n\in\supp(\mu_n)$ which implies that $x_n$ is a recurrent point.
	
	As we have assumed $\dim E^u(\sigma)=1$, let $\lambda_1,\cdots,\lambda_d$ be eigenvalues of $DX(\sigma)$ which satisfy
	$|{\rm Re}\lambda_1 |\leq \cdots \leq |{\rm Re}\lambda_{d-1}|
	< 1< |{\rm Re}\lambda_d|=|\lambda_d |$.
	We fix a constant $\epsilon_1$ satisfying $0<\epsilon_1<\min\big\{-\log|{\rm Re}\lambda_{d-1}|,\log|\lambda_d|\big\}$, and let
	\begin{align}\label{mu1}
	u_1=\log|{\rm Re}\lambda_{d-1}|+\epsilon_1<0.
	\end{align}
	
	\begin{claim}\label{clm:sigma-nbhd}
		There exists a neighborhood $U(\sigma)$ of $\sigma$ admitting a $C^1$-coordinate
		$$
		U(\sigma)~=~
		\left\{(x,y)=(x_1,\cdots,x_{d-1},y)\in\RR^{d-1}\times\RR:
		~\|x\|=\big(\sum_{i=0}^{d-1}x_i^2\big)^{1/2}\leq1,
		~|y|\leq1\right\}
		$$
		such that
		\begin{enumerate}
			\item The singularity $\sigma=(0,0)\in\RR^{d-1}\times\RR$, and the local stable and unstable manifolds of $\sigma$ in $U(\sigma)$ are
			$$
			W^s_{\it loc}(\sigma)=\big\{(x,0):~\|x\|\leq1\big\},
			\qquad {\it and} \qquad
			W^u_{\it loc}(\sigma)=\big\{(0,y):~|y|\leq1\big\}.
			$$
			\item For every $z=(x,y)\in U(\sigma)$, let $\phi_t(z)=\big(x(t),y(t)\big)$ be the orbit of $z$  under $\phi_t$ where $\big(x(0),y(0)\big)=(x,y)$. Assume the orbit segment $\phi_{[T_1,T_2]}(z)\subset U(\sigma)$ where $T_1<0<T_2$, then it satisfies $|y'(t)-\lambda_d\cdot y(t)|\leq\epsilon_1\cdot|y(t)|,~~\forall T_1\leq t \leq T_2$,
			$$
			\|x(t)\|\leq\exp(u_1 t)\|x(0)\|, ~~\forall 0\leq t \leq T_2;
			\quad {\it and} \quad
			\|x(t)\|\geq\exp(u_1 t)\|x(0)\|, ~~\forall T_1\leq t\leq0.
			$$
			\item $U(\sigma)\cap\orb(p,\phi_t)=\emptyset$, and for $z^{\pm}=(0,\pm1)\in\RR^{d-1}\times\RR$, they satisfy
			$\phi_t(z^\pm)\notin U(\sigma)$ for every $t>0$.      	
		\end{enumerate}
	\end{claim}
	
	\begin{proof}[Proof of Claim~\ref{clm:sigma-nbhd}]
		The first two items have been proved in Lemma 2.1, Theorem 2.4 and Theorem 2.5 of \cite{Shilnikov}. The third item comes from the fact that $z^\pm\in W^u_{\it loc}(\sigma)$ while their orbits $\orb(z^\pm)\subseteq W^s(\orb(p))$. So we only need to take $U(\sigma)$ small enough.
	\end{proof}
	
	For every $0<\delta<1$, we denote $U(\sigma,\delta)=\big\{(x,y)\in U(\sigma):\|x\|\leq\delta,|y|\leq\delta\big\}$. Let 
	$$
	\Sigma^{\pm}=\big\{(x,\pm1)\in U(\sigma):\|x\|\leq1\big\},
	\qquad {\rm and} \qquad
	\Sigma^{\pm}(\delta)=\big\{(x,\pm1)\in \Sigma^{\pm}:\|x\|\leq\delta\big\}.
	$$ 
	then $\Sigma^{\pm}\subseteq\partial U(\sigma)$. Moreover, the second item of Claim \ref{clm:sigma-nbhd} shows that the vector field $X$ is transverse to $\Sigma^{\pm}$ and points outside $U(\sigma)$. For every $\delta>0$, $\Sigma^\pm$ are local cross sections of $X$ centered at $z^\pm$ respectively.
	
	\begin{claim}\label{clm:time-estimate-1}
		For every $0<\delta<1$ and $z=(x_z,y_z)\in U(\sigma,\delta)$ with $x_z\neq 0,y_z>0$, it satisfies
		\begin{enumerate}
			\item  There exist $T_1<0<T_2$ such that $\phi_{T_1}(z)\in\partial U(\sigma)\setminus\Sigma^\pm$, 
			$\phi_{T_2}(z)\in\Sigma^+$, and $\phi_t(z)\in {\rm Int}U(\sigma)$ for every $t\in(T_1,T_2)$.
			\item There exist $T_1'\in(T_1,0]$ and $T_2'\in[0,T_2)$, such that $\phi_{[T_1,T_2]}(z)\cap U(\sigma,\delta)=\phi_{[T_1',T_2']}(z)$.
			\item If  $(x_w,1)=\big(x(T_2),y(T_2)\big)=\phi_{T_2}(z)\in\Sigma^+$, then we have
			$\|x_w\|<\delta<1$ and
			$$
			T_2'-T_1'~\leq~T_2-T_1~\leq~
			\frac{1}{u_1}\cdot\big[\log\|x_w\|-\log1\big]
			~=~\frac{1}{u_1}\cdot\log\|x_w\|.
			$$
		\end{enumerate}
		The same conclusion holds for $x_z\neq0,y_z<0$ if we replace $\Sigma^+$ by $\Sigma^-$.
	\end{claim}
	
	\begin{proof}[Proof of Claim~\ref{clm:time-estimate-1}]
		From the second item of Claim \ref{clm:sigma-nbhd}, for $z=(x,y)$ satisfying $x\neq0$ and $y>0$, the flow $\big(x(t),y(t)\big)$ satisfies
		\begin{itemize}
			\item $y(t)>0$ and monotonous increasing with respect to $t$, which implies $\exists T_2>0$ such that $\phi_{T_2}(z)\in\Sigma^+$;  
			\item $\|x(t)\|\geq\exp(u_1 t)\|x(0)\|$ for every $t\leq0$, which implies $\exists T_1<0$, such that $\phi_{T_1}(z)\in\partial U(\sigma)\setminus\Sigma^\pm$.
		\end{itemize}
		Moreover, we have $\phi_t(z)\in {\rm Int}U(\sigma)$ for every $t\in(T_1,T_2)$. This proves the first item. The proof of  second item is the same. 
		
		As $\|x(T_1)\|=1$ and $\|x(T_2)\|=\|x_w\|\leq\exp(u_1 T_2)\|x\|<\delta$,
		we apply the fact that $\|x(t)\|\geq\exp(u_1 t)\|x(0)\|$ for every $t\leq0$, which shows
		$$
		T_2'-T_1'~\leq~T_2-T_1~\leq~
		\frac{1}{u_1}\cdot\big(\log\|x_w\|-\log1\big)
		~=~\frac{1}{u_1}\cdot\log\|x_w\|.
		$$
	\end{proof}
	
	Recall $p$ is a hyperbolic periodic point of $X$ with index $0<\ind(p)<d-1$. We denote $k=\ind(p)$, then we have the following claim.
	
	\begin{claim}\label{clm:p-nbhd}
		There exist a cross section $\Sigma_p$ transverse to $X$ everywhere admitting a $C^1$ coordinate
		$$
		\Sigma_p=\left\{ (r,s)\in\RR^k\times\RR^{d-1-k}:
		\|r\|\leq1,\|s\|\leq1 \right\},
		$$
		and two constants $\kappa>1,~T_0>0$,
		such that
		\begin{enumerate}
			\item The orbit $\orb(p,\phi_t)\cap\Sigma_p=\{p\}$ and $p=(0,0)\in\RR^k\times\RR^{d-1-k}$.
			\item The local stable and unstable manifolds of the orbit of $p$ intersect $\Sigma_p$ as
			$$
			W^s_{\it loc}(\orb(p))\cap\Sigma_p=
			\big\{ (r,0):\|r\|\leq1 \big\},
			\qquad
			W^u_{\it loc}(\orb(p))\cap\Sigma_p=
			\big\{ (0,s):\|s\|\leq1 \big\}.
			$$
			\item The Poincar\'e return map 
			$$
			R:~\Sigma_p'=
			\big\{(r,s)\in\Sigma_p:\|s\|\leq\kappa^{-1}\big\}
			\longrightarrow \Sigma_p
			$$
			is well defined and $C^1$-smooth on $\Sigma_p'$. For every $(r_0,s_0)\in\Sigma_p'$, if we denote $(r_1,s_1)=R(r,s)$, then
			$$
			\|r_1\|<\|r_0\|, \qquad {\rm and} \qquad 
			\|s_1\|\leq\kappa\cdot\|s_0\|.
			$$
			\item For every $z\in\Sigma_p'$, let $T_z>0$ be the first return time of $z$ to $\Sigma_p$, i.e. $R(z)=\phi_{T_z}(z)\in\Sigma_p$, then $T_z\geq T_0$ and
			$$
			\left( \bigcup_{z\in\Sigma_p'}\phi_{[0,T_z]}(z)  \right)
			~\cap~ U(\sigma) ~=~\emptyset.
			$$
		\end{enumerate}
	\end{claim} 
	
	\begin{proof}[Proof of Claim~\ref{clm:p-nbhd}]
		Let $\pi(p)$ be the period of $p$ and $u_2>0$ be the largest Lyapunov exponent of $\orb(p)$. 
		There exists a $C^1$ cross section $\Sigma_1$ containing $p$ and transversing to $X$ everywhere. By shrinking $\Sigma_1$ if necessary, we can assume $\Sigma_1\cap\orb(p,\phi_t)=\{p\}$.
		
		Since $p$ is a hyperbolic periodic orbit of $X$, there exists a small neighborhood $\Sigma_2\subseteq\Sigma_1$, such that the Poincar\'e map $R:\Sigma_2\rightarrow\Sigma_1$ is well defined and $C^1$-smooth. Moreover, $p$ is a hyperbolic fixed point of $R$ with largest Lyapunov exponent equal to $\big(\pi(p)u_2\big)$. For every $z\in\Sigma_2$, let $T_z$ be the first return time of $z$ in $\Sigma_1$, i.e. $R(z)=\phi_{T_z}(z)$. 
		By shrink $\Sigma_2$ if necessary, there exists $T_0>0$ which is close to $\pi(p)$, such that $T_z\geq T_0$ for every $z\in\Sigma_2$. Moreover, since $U(\sigma)\cap\orb(p)=\emptyset$, we can assume
		$$
		\left( \bigcup_{z\in\Sigma_2}\phi_{[0,T_z]}(z)  \right)
		~\cap~ U(\sigma) ~=~\emptyset.
		$$
		
		Finally, we fix a small constant $\epsilon_2>0$ and take 
		\begin{align}\label{kappa}
		\kappa=\exp\big[\pi(p)\cdot(u_2+\epsilon_2)\big].
		\end{align}
		Similar to Claim \ref{clm:sigma-nbhd}, we can take an adapted $C^1$-coordinate on $\Sigma_2$, such that $p=(0,0)\in\RR^k\times\RR^{d-1-k}$ in this coordinate and the second and third items of the claim are satisfied on
		$$
		\Sigma_p=
		\left\{ (r,s)\in\RR^k\times\RR^{d-1-k}:
		\|r\|\leq1,\|s\|\leq1 \right\}
		~\subseteq~\Sigma_2.
		$$
		See \cite[Theorem 2.4 \& 2.5]{Shilnikov}. This finishes the proof of the claim.   	
	\end{proof}
	
	Let $V(p)=\bigcup_{z\in\Sigma_p'}\phi_{[0,T_z]}(z)$ which is a neighborhood of $\orb(p)$.
	Since $W^u(\sigma)\setminus\{\sigma\}\subseteq W^s(\orb(p))$, for $z^\pm=(0,\pm1)\in U(\sigma)\cap W^u_{\it loc}(\sigma)$,  there exist $T^\pm>0$ and $(r^\pm,0)\in{\rm Int}\big(\Sigma_p'\big)$ such that
	$$
	(r^+,0)=\phi_{T^+}(z^+), 
	\qquad {\rm and} \qquad
	(r^-,0)=\phi_{T^-}(z^-).
	$$
	Since $\orb^+(z^\pm)\cap U(\sigma)=\emptyset$, there exist two constants $\delta_0>0,K>0$ with $K\cdot\delta_0\ll1$, such that
	\begin{itemize}
		\item The Poincar\'e map
		$P^+:\Sigma^+(\delta_0)\rightarrow\Sigma_p$ is well defined and $C^1$-smooth, satisfying $P^+(z^+)=P^+(0,1)=(r^+,0)$ and for every $w=(x_w,1)\in\Sigma^+(\delta_0)$, if we denote $P^+(w)=(r_w,s_w)$, then it satisfies
		$$
		\|r_w-r^+\|\leq K\cdot\|x_w\|,
		\qquad {\rm and} \qquad
		\|s_w\|\leq K\cdot\|x_w\|.
		$$
		Moreover, if we denote $T^+_w>0$  the time satisfying $P^+(w)=\phi_{T^+_w}(w)$, then the orbit segment $\phi_{(0,T^+_w)}(w)\cap U(\sigma)=\emptyset$.
		
		\item The Poincar\'e map
		$P^-:\Sigma^-(\delta_0)\rightarrow\Sigma_p$ is well defined and $C^1$-smooth, satisfying $P^-(z^-)=P^-(0,-1)=(r^-,0)$ and for every $w=(x_w,-1)\in\Sigma^-(\delta_0)$, if we denote $P^-(w)=(r_w,s_w)$, then it satisfies
		$$
		\|r_w-r^-\|\leq K\cdot\|x_w\|,
		\qquad {\rm and} \qquad
		\|s_w\|\leq K\cdot\|x_w\|.
		$$
		Moreover, if we denote $T^-_w>0$  the time satisfying $P^-(w)=\phi_{T^-_w}(w)$, then the orbit segment $\phi_{(0,T^-_w)}(w)\cap U(\sigma)=\emptyset$.
	\end{itemize}
	
	\begin{claim}\label{clm:time-estimate-2}
		For every $z=(x_z,y_z)\in U(\sigma,\delta_0)$ with $x_z\neq0,y_z>0$, let $T_1<0<T_2$ be defined as in Claim \ref{clm:time-estimate-1} and $w=(x_w,1)=\phi_{T_2}(z)\in\Sigma^+(\delta_0)$, then for the point $P^+(w)=(r_w,s_w)\in\Sigma_p'$, let 
		$$
		T_p=\left(-\frac{\log(K\|x_w\|)}{\log\kappa}-2\right)\cdot T_0.
		$$ 
		we have $\phi_{[0,T_p]}(P^+(w))\subseteq V(p)$.
		
		In particular, combined with the third item of Claim \ref{clm:time-estimate-1} that $T_2-T_1\leq \log\|x_w\|/u_1$, we have
		\begin{align}\label{equ:time}
		\frac{T_p}{T_2-T_1}~\geq~
		\frac{-u_1\cdot T_0}{\log\kappa}\cdot
		\left(1+\frac{\log K+2\log\kappa}{\log\|x_w\|}\right).
		\end{align}
		
		The same conclusion holds for $z=(x_z,y_z)\in U(\sigma,\delta_0)$ with $x_z\neq0,y_z<0$.
	\end{claim}
	
	\begin{proof}[Proof of Claim~\ref{clm:time-estimate-2}]
		We only prove the case that $y_z>0$.
		From the property of Poincar\'e map $P^+:\Sigma^+(\delta_0)\to\Sigma_p$, we have $\|s_w\|\leq K\cdot\|x_w\|$. Moreover, for the Poincar\'e map $R:\Sigma_p'\to\Sigma_p$ in the neighborhood of ${\rm Orb}(p)$, we can apply the third item of Claim \ref{clm:p-nbhd} which shows that for the point $P^+(w)\in\Sigma_p'$, it will iterates by $R$ at least $([-\log\|s_w\|/\log\kappa]-1)$-times hitting $\Sigma_p'$ inside $V(p)$. So for
		$$
		T_p~=~\left(-\frac{\log(K\|x_w\|)}{\log\kappa}-2\right)\cdot T_0
		~\leq~
		\left(\left[\frac{\log\|s_w\|}{\log\kappa}\right]-1\right)\cdot T_0,
		$$
		we have $\phi_{[0,T_p]}(P^+(w))\subseteq V(p)$.
	
		Moreover, 
		by the fact $0<T_2-T_1\leq \log\|x_w\|/u_1$, we have the following estimates:
		\begin{align*}
			\frac{T_p}{T_2-T_1}&\geq \left(-\frac{\log(K\|x_w\|)}{\log\kappa}-2\right)\cdot T_0\cdot \frac{u_1}{\log\|x_w\|}\\
			&= \frac{-u_1\cdot T_0}{\log \kappa}\left(\frac{\log(K\|x_w\|)+2\log\kappa}{\log\|x_w\|}  \right)\\
			&= \frac{-u_1\cdot T_0}{\log\kappa}\cdot
			\left(1+\frac{\log K+2\log\kappa}{\log\|x_w\|}\right).
		\end{align*}
		This proves the claim.
	\end{proof}
	
	Now we take $\delta_1\in(0,\delta_0]$ small enough, such that
	$(\log K+2\log\kappa)/\log\delta_1>-1/2$.
	Then for every $z=(x_z,y_z)\in U(\sigma,\delta_1)$, the corresponding estimation of Equation (\ref{equ:time}) is
	\begin{align}\label{equ:time-1}
	T_p~\geq~
	L\cdot(T_2-T_1),
	\qquad {\rm where} \qquad
	L=\frac{-u_1\cdot T_0}{2\log\kappa}.
	\end{align}

	Now we take a $C^1$ smooth function $\varphi:M\rightarrow\RR$ satisfying:
	\begin{enumerate}
		\item $\varphi(z)\geq-1$ for every $z\in M$ and $\varphi(\sigma)=-1$.
		\item $\varphi(z)\geq0$ for every $z\in M\setminus U(\sigma,\delta_1)$.
		\item $\varphi(z)\geq L^{-1}$ for every $z\in V(p)$.
	\end{enumerate}
	Then for   the generic point $x_n$ of $\mu_n$, let 
	$$
	0~\leq~ T_1^1<T_2^1<T_3^1<T_4^1~<~
	\cdots\cdots
	~<~T_1^m<T_2^m<T_3^m<T_4^m~<~\cdots\cdots
	$$
	be the time sequence where $\{T_1^m\}_m$ are all the times that $\phi_t(x_n)$ enters $U(\sigma,\delta_1)$ and $T_2^m$ is the time when $\phi_t(x_n)$ enters $U(\sigma,\delta_1)$ after $T_1^m$; $T_3^m$ is the time when $\phi_t(x_n)$ enters $V(p)$ after $T_2^m$ and $T_4^m$ is the time  when $\phi_t(x_n)$ leaves $V(p)$ after $T_3^m$. 
	
	For every $m$, we have 
	$$
	(T_4^m-T_3^m)~\geq~L\cdot (T_2^m-T_1^m).
	$$
	This implies
	\begin{align*}
	\int\varphi{\rm d}\mu_n~& =~
	\lim_{T\rightarrow+\infty}\frac{1}{T}
	\int_{0}^{T}\varphi\big(\phi_t(x_n)\big){\rm d}t \\
	& \geq~
	\lim_{T_4^m\rightarrow+\infty}
	\frac{1}{T_4^m}\sum_{i=1}^{m}\left[(-1)\cdot(T_2^m-T_1^m)+ L^{-1}\cdot(T_4^m-T_3^m)\right] \\
	& \geq~0.
	\end{align*}
	However, we have $\int\varphi{\rm d}\delta_{\sigma}=\varphi(\sigma)=-1$. This contradicts to $\mu_n\rightarrow\delta_\sigma$, thus $\delta_\sigma$ is isolated in $\mathcal{M}_{erg}(X)$.
\end{proof}

\subsection{Proof of Theorem~\ref{Thm:B} and Corollary~\ref{Cor:Cr-entropy-support}}\label{Section:ergodic-measure-Lorenz}

We split the proof of Theorem~\ref{Thm:B} into two parts: the dense part and the residual part. We first state  two lemmas.
Recall by Proposition~\ref{prop:Lorenz}, any geometric Lorenz attractor of a vector field $X\in\mathscr{X}^r(M^3)$ is a homoclinic class and any two periodic orbits are homoclinically related. Hence we have the following lemma which is a combination of~\cite[Proposition 4.7 $\&$ Remark 4.6]{abc}.

\begin{lemma}\label{Lem:convex of periodic measure}
	Assume $r\in\mathbb{N}_{\geq 2}\cup\{\infty\}$ and $X\in\mathscr{X}^r(M^3)$ admits a geometric Lorenz attractor $\Lambda$, then $\overline{\mathcal{M}_{per}(\Lambda)}$ is convex.
\end{lemma}

Next lemma states that periodic measures is dense in the ergodic measure space of Lorenz attractors for $C^r$-generic $X\in\mathscr{X}^r(M^3)$.  Its proof is by a classical Baire argument, which we postpone in the end of this section.
\begin{lemma}\label{Lem:density in erg}
	For every $r\in\mathbb{N}_{\geq 2}\cup\{\infty\}$, there exists a $C^r$-residual subset $\mathcal{R}^r\subset \mathscr{X}^r(M^3)$ such that if $X\in\mathcal{R}^r$ admits a Lorenz attractor $\Lambda$, then $\mathcal{M}_{erg}(\Lambda)\subset\overline{\mathcal{M}_{per}(\Lambda)}$.
\end{lemma}

Now we manage to prove  Theorem~\ref{Thm:B} and Corollary~\ref{Cor:Cr-entropy-support}.
In certain cases, we have to perturb a vector field in different open regions. The results in~\cite[Section 2]{PR} guarantee that every vector field $X\in\mathscr{X}^r(M)$ ($r\in\mathbb{N}\cup\{\infty\}$) admits a basis of neighborhoods $\mathcal{U}$ satisfying the following property:
\begin{itemize}
	\item[$(F)$]: {\it For any two perturbations $Y_1,Y_2\in \mathcal{U}$ of $X$, if there exists two open sets $W_1,W_2$ such that $Y_i=X|_{M\setminus W_i}$ for $i=1,2$ and $W_1\cap W_2=\emptyset$, then the composed perturbation $Y$ is also in $\mathcal{U}$, where $Y=X|_{M\setminus {W_1\cup W_2}}$ and $Y=Y_i|_{W_i}$ for $i=1,2$.}
\end{itemize}

\begin{proof}[Proof of Theorem~\ref{Thm:B}]	
	Let $r\in\mathbb{N}_{\geq 2}\cup\{\infty\}$. We splits the proof of Theorem~\ref{Thm:B} into dense part and residual part.
   \paragraph{Dense part:}
	By Proposition~\ref{prop:Lorenz} and robustness of hyperbolic singularities, one can easily prove that there exists an open and dense subset $\mathcal{O}$ in $\mathscr{X}^r(M^3)$ such that the number of geometric Lorenz attractors of every $X\in\mathcal{O}$ is robustly constant. That is to say, for any $X\in\mathcal{O}$, there exists a $C^r$ neighborhood $\mathcal{V}_X$ of $X$ such that the number of geometric Lorenz attractors of any $Z\in \mathcal{V}_X$ is a constant which equals that of $X$.
	
	To prove the dense part, we only need to show that for any $X\in\mathcal{O}$ and for any $C^r$ neighborhood $\mathcal{V}$ of $X$, there exists a vector field $Y\in\mathcal{V}$ such that every geometric Lorenz attractor $\Lambda_Y$ of $Y$ satisfies 
	$$\overline{\mathcal{M}_{per}(\Lambda_Y)}\subsetneqq \overline{\mathcal{M}_{erg}(\Lambda_Y)}\subsetneqq \mathcal{M}_{inv}(\Lambda_Y). $$
	
	By hyperbolicity of singularities, we have that $X$ admits  only finitely many geometric Lorenz attractors $\Lambda_1,\Lambda_2,\cdots,\Lambda_k$ with pairwise disjoint attracting regions $U_1,U_2,\cdots,U_k$. Shrinking $\mathcal{V}$ if necessary, we assume that every $Z\in\mathcal{V}$ admits exactly $k$ geometric Lorenz attractors $\Lambda_{1,Z},\Lambda_{2,Z},\cdots,\Lambda_{k,Z}$ with the same attracting regions $U_1,U_2,\cdots,U_k$. Moreover, we assume that $\mathcal{V}$ satisfies the property $(F)$ above. 
	For each $Z\in\mathcal{V}$ and for each $i=1,2,\cdots,k$, we denote by $\sigma_i^{Z}$ the unique singularity in $\Lambda_{i,Z}$ and denote by $\sigma_i=\sigma_i^{X}$ for simplicity. 
	We first show the following claim.
	
	\begin{claim}\label{Claim:Y_i}
		For each $i=1,2,\cdots,k$, there exists $Y_i\in\mathcal{V}$ satisfying that $Y_i|_{M\setminus U_i}=X|_{M\setminus U_i}$ and $W^u(\sigma_i^{Y_i},\phi_t^{Y_i})\setminus \{\sigma_i^{Y_i}\}\subset W^s(\orb(p_i^{Y_i},\phi_t^{Y_i}))$ where $p_i^{Y_i}$ is a hyperbolic periodic point of $\phi_t^{Y_i}$.
	\end{claim}
    \begin{proof}[Proof of Claim~\ref{Claim:Y_i}]
    	Fix $i\in\{1,2,\cdots,k\}$. We denote by $\Lambda=\Lambda_i,U=U_i$ and $\sigma=\sigma_i$ to simplify notations.
    	
    	Let $\Sigma\subset U$ be the cross section of $\Lambda$ and $z^+,z^-\in W^u(\sigma)\cap\Sigma$ be the two different points satisfying $\phi_t^X(z^+),\phi_t^X(z^-)\notin\Sigma$ for any $t<0$. 
    	Take a hyperbolic periodic point $p\in\Lambda$. 
    	Take a constant $\delta^+>0$ such that 
    	\begin{itemize}
    	    \item $B(z^{+},\delta^+)\subset U\setminus (\orb(p,\phi^X_t)\cup\{\sigma\})$, 
    	    \item $\phi^X_t(z^-)\notin B(z^{+},\delta^+)$ for any $t\leq 0$.
    	\end{itemize}
    	
    	By Theorem~\ref{thm:Lorenz-connecting}, there exists $Y^+\in\mathcal{V}$ satisfying that 
    	$$Y^+|_{M\setminus B(z^{+},\delta^+)}=X|_{M\setminus B(z^{+},\delta^+)} \text{~~and~~} \orb(z^+,\phi_t^{Y^+})\subset W^u(\sigma,\phi_t^{Y^+})\cap W^s(\orb(p),\phi_t^{Y^+}).$$
    	
        Note that $\mathcal{V}$ is also a $C^r$-neighborhood of $Y^+$ with $z^-\in W^u(\sigma,\phi^{Y^+}_t)$ and $p\in\per(\phi^{Y^+}_t)$. Moreover, by the choice of $\mathcal{V}$, the maximal invariant compact set of $\phi^{Y^+}_t$ is still a geometric Lorenz attractor with singularity $\sigma$ and containing $\orb(p)$ in it. 
        We take another constant $\delta^->0$ satisfying that 
        \begin{itemize}
        	\item $B(z^{-},\delta^-)\subset U\setminus (\orb(p,\phi^X_t)\cup\{\sigma\})$, 
        	\item $\phi^{Y^+}_t(z^+)\notin B(z^{-},\delta^-)$ for any $t\in\mathbb{R}$.
        \end{itemize}
    	Applying Theorem~\ref{thm:Lorenz-connecting} again for the vector field $Y^+$ and the neighborhood $\mathcal{V}$, there exists $Y_i\in\mathcal{V}$ such that 
    	$$Y_i|_{M\setminus B(z^{-},\delta^-)}=Y^+|_{M\setminus B(z^{-},\delta^-)} \text{~~and~~} \orb(z^-,\phi_t^{Y_i})\subset W^u(\sigma,\phi_t^{Y_i})\cap W^s(\orb(p),\phi_t^{Y_i}).$$
    	By the choice of $\delta^-$, we know that it satisfies also 
    	$$\orb(z^+,\phi_t^{Y_i})\subset W^u(\sigma,\phi_t^{Y_i})\cap W^s(\orb(p),\phi_t^{Y_i}).$$
    	The vector field $Y_i$ satisfies Claim~\ref{Claim:Y_i}.
    \end{proof}
	
	Finally, let $Y$ be the composed perturbation of $X$ satisfying that 
	$$Y|_{U_i}=Y_i|_{U_i}, ~\text{ for each~} i=1,2,\cdots,k,~~~\text{ and~~~}  Y|_{M\setminus (\bigcup_{1\leq i\leq k}U_i)}=X|_{M\setminus (\bigcup_{1\leq i\leq k}U_i)}.$$
	By the property $(F)$, we have that $Y\in\mathcal{V}$ since the attracting regions $\{U_i\}_{1\leq i\leq k}$ are pairwise disjoint. By Proposition~\ref{Prop:kick-out singularity}, for each $1\leq i\leq k$, the Dirac measure $\delta_{\sigma_i}=\delta_{\sigma_i^Y}$ is isolated in $\mathcal{M}_{erg}(Y)$.  In particular, we have 
	$$\overline{\mathcal{M}_{per}(\Lambda_{i,Y})}\subsetneqq \overline{\mathcal{M}_{erg}(\Lambda_{i,Y})}\subsetneqq \mathcal{M}_{inv}(\Lambda_{i,Y}). $$
	The proof of dense part is completed.
	
	\paragraph{Residual part:}
	Let $\mathcal{R}^r\subset\mathscr{X}^r(M^3)$ be the residual subset from Lemma~\ref{Lem:density in erg}. Assume $X\in\mathcal{R}^r$ admits a geometric Lorenz attractor $\Lambda$. 
	We first show that $$\mathcal{M}_{inv}(\Lambda)=\overline{\mathcal{M}_{erg}(\Lambda)}=\overline{\mathcal{M}_{per}(\Lambda)}.$$
	
	Take $\mu\in \mathcal{M}_{inv}(\Lambda)$. 
	By the Ergodic Decomposition Theorem, there exists a sequence of measures $\mu_n$ converging to $\mu$ in the weak*-topology where $\mu_n\in\mathcal{M}_{inv}(\Lambda)$ is a convex sum of finitely many ergodic measures supported on $\Lambda$ for each $n\in\mathbb{N}$. By Lemma~\ref{Lem:density in erg}, there exists a sequence of measures $\nu_n$ converging to $\mu$ in the weak*-topology where $\nu_n\in\mathcal{M}_{inv}(\Lambda)$ is a convex sum of finitely many periodic measures supported on $\Lambda$ for each $n\in\mathbb{N}$.
	By Lemma~\ref{Lem:convex of periodic measure},  we have that $\overline{\mathcal{M}_{per}(\Lambda)}$ is convex. Hence $\mu\in\overline{\mathcal{M}_{per}(\Lambda)}$. This implies $\mathcal{M}_{inv}(\Lambda)=\overline{\mathcal{M}_{erg}(\Lambda)}=\overline{\mathcal{M}_{per}(\Lambda)}$.

	Note that $\mathcal{M}_{erg}(\Lambda)$ coincides with $\overline{\mathcal{M}_{per}(\Lambda)}\cap \mathcal{M}_{erg}(\Lambda)$. Hence the path connectedness of $\mathcal{M}_{erg}(\Lambda)$ follows directly from Proposition~\ref{prop:Lorenz} and Proposition~\ref{Prop:connectedness}. 
	This completes the residual part and hence the proof of Theorem~\ref{Thm:B}.
\end{proof}

\begin{proof}[Proof of Corollary~\ref{Cor:Cr-entropy-support}]
	Let $\mathcal{R}^r$ be the residual subset of $\mathscr{X}^r(M^3)$ from Theorem~\ref{Thm:B}.
	Assume  $X\in \mathcal{R}^r$ admits a geometric Lorenz attractor $\Lambda$. By~\cite[Proposition 5.1]{abc}, $\mathcal{M}_{erg}(\Lambda)$ is a $G_{\delta}$ subset in $\mathcal{M}_{inv}(\Lambda)$. Since $\mathcal{M}_{inv}(\Lambda)=\overline{\mathcal{M}_{erg}(\Lambda)}$, we have that $\mathcal{M}_{erg}(\Lambda)$ is a residual subset in $\mathcal{M}_{inv}(\Lambda)$. On the other hand, periodic orbits are dense in $\Lambda$ since $\Lambda$ is a homoclinic class. Then by~\cite[Proposition 5.3]{abc}, the set of invariant measures with support equal to $\Lambda$ is a residual subset $\mathcal{M}_1$ in $\mathcal{M}_{inv}(\Lambda)$. By~\cite{PYY}, the entropy map 
	$$h_{(\cdot)}\colon \mathcal{M}_{inv}(\Lambda)\rightarrow \mathbb{R}$$
	$$~~~~~~~~~~\mu\mapsto h_{\mu}(X)$$
	is upper semi-continuous. Hence the continuity points of the entropy map $h_{(\cdot)}$ form a residual subset $\mathcal{M}_2$ in $\mathcal{M}_{inv}(\Lambda)$. Moreover, since $\mathcal{M}_{inv}(\Lambda)=\overline{\mathcal{M}_{per}(\Lambda)}$ and every periodic measure has $0$-entropy, we have that every $\mu\in\mathcal{M}_2$ satisfies $h_{\mu}(X)=0$.
	Let $$\mathcal{M}_{res}=\mathcal{M}_{erg}(\Lambda)\cap \mathcal{M}_1\cap \mathcal{M}_2.$$ 
	Then $\mathcal{M}_{res}$ is a residual subset in $\mathcal{M}_{inv}(\Lambda)$ and the residual part  concludes.
	
	To prove the dense part, since $\mathcal{M}_{inv}(\Lambda)=\overline{\mathcal{M}_{per}(\Lambda)}$, we only need to show that any periodic measure supported on $\Lambda$ is approximated by ergodic measures in $\mathcal{M}_{erg}(\Lambda)$ with positive entropy. 
	The idea follows~\cite[Theorem 3.5 part (c)]{abc} and we give an explanation here. 
	For any periodic orbit $\orb(p)$ in $\Lambda$, we have that $\Lambda=H(p)$ by Proposition~\ref{prop:Lorenz}. Hence $\orb(p)$ exhibits some transverse homoclinic orbit which implies that there are hyperbolic horseshoes contained in $\Lambda$ arbitrarily close to this homoclinic orbit. That means points in the hyperbolic horseshoes spends arbitrarily large fractions of their orbits shadowing $\orb(p)$ as closely as required. Hence measures supported on such hyperbolic horseshoes are close to $\delta_{\orb(p)}$ in the weak*-topology. On the other hand, every hyperbolic horseshoe must support an ergodic measure with positive entropy. Thus $\delta_{\orb(p)}$ is approximated by ergodic measures in $\mathcal{M}_{erg}(\Lambda)$ with positive entropy.
\end{proof}

\begin{remark}\label{Rem:entropy and support}
	In the proof of Corollary~\ref{Cor:Cr-entropy-support}, the residual subsets $\mathcal{M}_1$ and $\mathcal{M}_2$ of $\mathcal{M}_{inv}(\Lambda)$ exist for every geometric Lorenz attractor $\Lambda$ of every $X\in\mathscr{X}^r(M^3)$. Therefore the measures with zero entropy and full support form a residual subset $\mathcal{M}_0$ in $\mathcal{M}_{inv}(\Lambda)$ for every geometric Lorenz attractor $\Lambda$. Here is an explanation:
	
	\noindent First, periodic points are dense in $\Lambda$ since it is a homoclinic class by Proposition~\ref{prop:Lorenz}. By ~\cite[Proposition 5.3]{abc}, the set of invariant measures with full support  is a residual subset $\mathcal{M}_1$ in $\mathcal{M}_{inv}(\Lambda)$. Second, by Proposition~\ref{Prop:dense periodic measure},  the set $\mathcal{M}_{per}(\Lambda)\cup \{\delta_{\sigma}\}$ is dense in $\mathcal{M}_{erg}(\Lambda)$. This implies that every $\mu\in\mathcal{M}_{inv}(\Lambda)$ is the limit of invariant measures $\{\mu_n\}$ where each $\mu_n$ is the convex sum of finitely many periodic measures  supported on $\Lambda$ together with $\delta_{\sigma}$. Thus by the upper semi-continuity of the entropy map~\cite{PYY}, the invariant measures with zero entropy form a residual subset $\mathcal{M}_2$ in $\mathcal{M}_{inv}(\Lambda)$. 
\end{remark}

To end this section, we prove Lemma~\ref{Lem:density in erg}.

\begin{proof}[Proof of Lemma~\ref{Lem:density in erg}]
	Assume $X\in\mathscr{X}^r(M^3)$ admits a Lorenz attractor $\Lambda$ with attracting region $U$ and let $\mathcal{U}$ be the $C^r$ neighborhood of $X$ satisfying Proposition~\ref{prop:Lorenz} for $\Lambda$ and $U$.
	Recall that  $\Lambda$ is a homoclinic class and any two periodic orbits are homoclinically related by Proposition~\ref{prop:Lorenz}.
	Applying Corollary~\ref{Cor:dense-loop} and Proposition~\ref{Prop:loop}, we have the following claim.
	\begin{claim}\label{Claim:approach singularity}
		For any $C^r$-neighborhood $\mathcal{V}$ of $X$ and any $\varepsilon>0$, there exists $Y\in\mathcal{V}$ which admits a periodic point $p$ satisfying that  $d(\delta_{\orb(p,\phi^Y_t)},\delta_{\sigma_Y})<\varepsilon$.
	\end{claim}

	Lemma~\ref{Lem:density in erg} concludes by a standard Baire argument. Take a countable basis $\{O_n\}_{n\in\mathbb{N}^+}$ of $M^3$ and let $\{U_n\}_{n\in\mathbb{N}^+}$ be the countable family which consists of all the possible unions of finitely many elements in $\{O_n\}_{n\in\mathbb{N}^+}$. For each pair of integers $n,m\in\mathbb{N}^+$, we define the following two subsets of $\mathscr{X}^r(M^3)$:
	\begin{itemize}
		\item $\mathcal{F}_{m,n}$: $Y\in \mathcal{F}_{m,n}$ if and only if $Y$ admits a hyperbolic singularity $\sigma_Y$ in $U_m$ and there exists a periodic point $q_Y$ of $Y$ such that $d(\delta_{\orb(q_Y,\phi^Y_t)},\delta_{\sigma_Y})<\frac{1}{n}$;
		
		\item $\mathcal{N}_{m,n}$: $Y\in \mathcal{F}_{m,n}$ if and only if there exists a neighborhood $\mathcal{V}_Y\subset\mathscr{X}^r(M^3)$ of $Y$ such that for any $Z\in\mathcal{V}_Y$, one of the following cases satisfies
		\begin{itemize} 
			\item[--] either $Z$ admits no hyperbolic singularity in $U_m$;
			\item[--] or for any singularity $\sigma_Z$ of $Z$ in $U_m$ and for any periodic point $q_Z$ of $Z$, it satisfies $d(\delta_{\orb(q_Z,\phi^Z_t)},\delta_{\sigma_Z})\geq\frac{1}{n}$.
		\end{itemize}
	\end{itemize}
	
	\begin{claim}\label{Claim:Cr-open-dense}
		$\mathcal{F}_{m,n}\cup \mathcal{N}_{m,n}$ is open and dense in $\mathscr{X}^r(M^3)$.
	\end{claim}
	\begin{proof}
		By the robustness of hyperbolic critical elements, we have that $\mathcal{F}_{m,n}$ is open in $\mathscr{X}^r(M^3)$. The openness of $\mathcal{N}_{m,n}$ is by definition.
		
		For any vector field $Y\in\mathscr{X}^r(M^3)\setminus\mathcal{N}_{m,n}$, there exists a sequence of vector fields $Y_i\in\mathscr{X}^r(M^3)$ such that $Y_i$ admits a hyperbolic singularity $\sigma_{Y_i}$ in $U_m$ and there exists a periodic point $q_{Y_i}$ of $Y_i$ such that $d(\delta_{\orb(q_{Y_i})},\delta_{\sigma_{Y_i}})<\frac{1}{n}$. Again by the robustness of hyperbolic critical elements, we have that $Y_i\in \mathcal{F}_{m,n}$ which implies that $Y\in \overline{\mathcal{F}_{m,n}}$. This shows that $\mathcal{F}_{m,n}\cup \mathcal{N}_{m,n}$ is dense in $\mathscr{X}^r(M^3)$.
	\end{proof}
	
	By Claim~\ref{Claim:Cr-open-dense}, the set  
	$$\mathcal{R}^r=\bigcap_{m,n\in\mathbb{N}^+}(\mathcal{F}_{m,n}\cup \mathcal{N}_{m,n})$$
	is a residual subset in $\mathscr{X}^r(M^3)$. 
	Let $X\in\mathcal{R}^r$ and $\Lambda$ be a Lorenz-like attractor of $X$ with singularity $\sigma$. We show that $\mathcal{M}_{erg}(\Lambda)\subset \overline{\mathcal{M}_{per}(\Lambda)}$.
	
	Take $\mu\in \mathcal{M}_{erg}(\Lambda)$. If $\mu\neq\delta_{\sigma}$, then $\mu\in\overline{\mathcal{M}_{per}(\Lambda)}$ by Proposition~\ref{Prop:dense periodic measure}.  Otherwise $\mu=\delta_\sigma$. By the hyperbolicity of $\sigma$, there exist a neighborhood $U_\sigma$ of $\sigma$ and a neighborhood $\mathcal{U}_X\subset\mathscr{X}^r(M^3)$ of $X$ such that any $Y\in\mathcal{U}_X$ admits a unique singularity $\sigma_Z$ in $U_\sigma$ where $\sigma_Z$ is the continuation of $\sigma$. Take $m\in\mathbb{N}^+$ such that $U_m\subset U_{\sigma}$ is a neighborhood of $\sigma_Z$ for all $Z\in \mathcal{U}_X$. By Claim~\ref{Claim:approach singularity}, we have that $X\notin \mathcal{N}_{m,n}$ for any $n\in\mathbb{N}^+$. Hence $X\in \mathcal{F}_{m,n}$ for any $n\in\mathbb{N}^+$. This implies that there exists a sequence of periodic points $p_n$ of $X$ such that $\delta_{\orb(p_n)}$ converges to $\delta_{\sigma}$ in the weak*-topology. Moreover, since $\Lambda$ is a Lorenz attractor, we have that $p_n\in\Lambda$ for each $n$ large enough. This shows that $\delta_{\sigma}\in \overline{\mathcal{M}_{per}(\Lambda)}$ and completes the proof of Lemma~\ref{Lem:density in erg}.
\end{proof}

\appendix
\section{Appendix: Robustness of geometric Lorenz attractors}\label{Section:robust}

In appendix~\ref{Section:robust}, we prove Proposition \ref{prop:Lorenz}, which shows the geometric Lorenz attractor is a homoclinic class and the vector fields exhibiting a geometric Lorenz attractor forms a $C^2$-open subset in $\mathscr{X}^r(M^3)$ for every $r\in\mathbb{N}_{\geq 2}\cup\{\infty\}$.

\begin{proof}[Proof of Proposition \ref{prop:Lorenz}]
	The proof of $\Lambda$ is a homoclinic class and every two periodic orbits of $\Lambda$ are homoclinic relies on the fact that the one-dimensional quotient map $f$ is eventually onto (Lemma \ref{lem:evenually-onto}). We refer to \cite[Theorem 1]{lorenz-homoclinic} for showing that $\Lambda$ is a homoclinic class. For every two periodic points $p,q\in\Lambda$, we can assume $p,q\in\Sigma$. Let $\gamma^u\subset W^u_{\it loc}(\orb(p),\phi_t^X)\cap\Sigma$ be the segment containing $p$ in $\Sigma$. Let $q=(x_q,y_q)\in\Sigma$, then $\{x_q\}\times[-1,1]\subset W^s_{\it loc}(\orb(q),\phi_t^X)$. We only need to show there exists some $n>0$ such that 
	$$
	P^n(\gamma^u)~\cap~\{x_q\}\times[-1,1]~\neq~\emptyset.
	$$
	Here $P:\Sigma\setminus l\rightarrow\Sigma$ is the Poincar\'e map. Let $J=\pi_x(\gamma^u)$. This is equivalent to there exists $n>0$ such that $x_q\in f^n(J)$, which has been proved in Lemma \ref{lem:evenually-onto}.
	
	\vskip1mm
	
	Now we prove the second part of this proposition. We only need to show for every $X\in\mathscr{X}^2(M^3)$ which exhibits a geometric Lorenz attractor with attracting region $U$, there exists a $C^2$-neighborhood $\cU$ of $X$, such that for every $Y\in\cU$, the maximal invariant set $\Lambda_Y=\bigcap_{t>0}\phi_t^Y(U)$ is a geometric Lorenz attractor.
	
	Let $\lambda_0>\sqrt{2}$ be the constant fixed in Notation \ref{notat:section} satisfying $f'(x)>\lambda_0$ for every $x\in[-1-\epsilon,1+\epsilon]\setminus\{0\}$. Let $\eta>0$ be a fixed constant satisfying 
	\begin{align}\label{equ:eta}
	(1+\eta)^{-6}\cdot\lambda_0>\sqrt{2}
	\qquad {\rm and} \qquad
	\sup_{(x,y)\in\Sigma\setminus l}\big|\partial H(x,y)/\partial x\big|<(1+\eta)^{-6}.
	\end{align}

	Since $U\subseteq M^3$ is an attracting region of $X$ and$\Lambda=\bigcap_{t>0}\phi_t^X(U)$ is singular hyperbolic $T_{\Lambda}M^3=E^{ss}\oplus E^{cu}$, there exists a $C^2$-neighborhood $\cU_1$ of $X$, such that for every $Y\in\cU_1$, $U$ is an attracting region of $Y$ and $\Lambda_Y=\bigcap_{t>0}\phi_t^Y(U)$ is singular hyperbolic with splitting $T_{\Lambda_Y}M^3=E^{ss}_Y\oplus E^{cu}_Y$. Moreover, The singularity $\sigma_Y\in\Lambda_Y$ also has three eigenvalues $\lambda_{1,Y}<\lambda_{2,Y}<0<\lambda_{3,Y}$ satisfying $\lambda_{1,Y}+\lambda_{3,Y}<0$ and $\lambda_{2,Y}+\lambda_{3,Y}>0$ for every $Y\in\cU_1$.

	Let $\Sigma=[-1-\epsilon,1+\epsilon]^2$ be the cross section of $X$ and $\Lambda$ defined in Lemma \ref{lem:big-section}.
	By the continuity of the singularity $\sigma_Y$ and the continuity of the local stable and unstable manifolds of $\sigma_Y$, shrinking $\cU_1$ if necessary, we can further assume that for every $Y\in\cU_1$, it satisfies
	\begin{enumerate}
		\item For every $z\in U\setminus W^s_{\it loc}(\sigma_Y,\phi_t^Y)$, there exists $T>0$, such that $\phi_T^Y(z)\in\Sigma$.
		\item The local stable manifold of $\sigma_Y$ intersects $\Sigma$ with $l_Y=W^s_{\it loc}(\sigma_Y,\phi_t^Y)\cap\Sigma$ being $C^2$-close to $l=\{0\}\times[-1-\epsilon,1+\epsilon]$.
		\item The unstable manifold of $\sigma_Y$ intersects $\Sigma$ with $z_Y^+,z_Y^-\in\Sigma$ which are close to $z^+,z^-$ respectively.
		\item The Poincar\'e map $P_Y:\Sigma\setminus l_Y\rightarrow\Sigma$ is well defined and satisfies
		$$
		\overline{P_Y(\Sigma\setminus l_Y)}~=~
		P_Y(\Sigma\setminus l_Y)\cup\{z_Y^+,z_Y^-\}~\subseteq~
		{\rm Int}(\Sigma).
		$$
		\item For every $(x,y)\in\Sigma\setminus l_Y$, we have
		$$
		\frac{\partial\pi_x\circ P_Y}{\partial x}(x,y)>(1+\eta)^{-1}\lambda_0, \quad
		\frac{\partial\pi_y\circ P_Y}{\partial x}(x,y)<(1+\eta)^{-3}, 
		\quad {\rm and} \quad
		\frac{\partial\pi_y\circ P_Y}{\partial y}(x,y)<1.
		$$
	\end{enumerate}
	
	For every $Y\in\cU_1$, we denote $\cF^{ss}_Y$ the strong stable foliation of $Y$ tangent to $E^{ss}_Y$ in $U$, and $\cF^s_Y=\bigcup_{t\in\RR}\phi_t^Y(\cF^{ss}_Y)\cap U$ the 2-dimensional stable foliation in $U\setminus\cF^{ss}_Y(\sigma_Y)$ defined as Lemma \ref{lem:stable}. Let 
	$$
	\cF^s_{Y,\Sigma}~=~\cF^s_Y~\cap~\Sigma
	$$
	be the stable foliation on $\Sigma$ induce by $\cF^s_Y$. 
	By the classical stable manifold theorem \cite{hps},
	for every $z=(x_z,y_z)\in\Sigma$, the leaf $\cF^s_{Y,\Sigma}(z)$ converges to $\{x_z\}\times[-1-\epsilon,1+\epsilon]$ in $C^2$-topology as $Y\rightarrow X$ in $\mathscr{X}^2(M^3)$. To be precise, shrinking $\cU$ if necessary, for every $x_z\in[-1-\epsilon/2,1+\epsilon/2]$, there exists a $C^2$-smooth function $\varphi^s_z:[-1-\epsilon,1+\epsilon]\rightarrow(-\epsilon/2,\epsilon/2)$, such that
	$$
	\cF^s_{Y,\Sigma}(z)=
	{\rm Graph}\big(\varphi^s_z\big)=
	\big\{x_z+\varphi^s_z(y):~y\in[-1-\epsilon,1+\epsilon]\big\}.
	$$
	The function $\varphi^s_z$ satisfies $\varphi^s_z(y_z)=0$ and $\varphi^s_z\rightarrow 0$ in $C^2$-topology as $Y\rightarrow X$ in $\mathscr{X}^2(M^3)$. 
	
	Let $h^s_{Y,\Sigma}$ be the holonomy map induced by $\cF^s_{Y,\Sigma}$ on $\Sigma$. Now we show $h^s_{Y,\Sigma}$ converges to the holonomy map of the foliation
	$\big\{\{x\}\times[-1-\epsilon,1+\epsilon]:~x\in[-1-\epsilon,1+\epsilon]\big\}$ in $C^1$-topology.
	
	\begin{claim}\label{clm:holonomy}
		The holonomy map $h^s_{Y,\Sigma}$ is $C^{1+}$-smooth on $\Sigma$. There exists a $C^2$-neighborhood $\cU_2$ of $X$, such that for every $Y\in\cU_2$ and every two points $z_i=(x_i,y_i)\in\cF^s_{Y,\Sigma}(z)$, $i=1,2$, let 
		$$
		J_i=(a_i,b_i)\times\{y_i\}~\subset~ [-1-\epsilon/2,1+\epsilon/2]\times[-1-\epsilon,1+\epsilon]
		$$
		be two intervals containing $x_i\in(a_i,b_i)$,
		such that $h^s_{Y,\Sigma}(z_1)=z_2$ and $h^s_{Y,\Sigma}(J_1)=J_2$, then
		$$
		(1+\eta)^{-1}~<~
		Dh^s_{Y,\Sigma}(z_1)
		~<~(1+\eta).
		$$
		This implies $h^s_{Y,\Sigma}$ converges to $h^s_{X,\Sigma}\equiv{\rm Id}_x$ in $C^1$-topology as $Y$ converges to $X$ in $\mathscr{X}^2(M^3)$.
	\end{claim}
	
	\begin{proof}[Proof of the claim]
		Due to a beautiful observation in \cite[Lemma 7.1]{am2},
		the holonomy map $h^s_Y$ of the stable foliation $\cF^s_{Y,\Sigma}$ is $C^{1+}$-smooth in the cross section $\Sigma$. The key fact is that $h^s_Y$ is absolutely continuous with respect to the transversal Lebesgue measure. Thus the fact that $\cF^s_{Y,\Sigma}$ is codimension one implies $h^s_Y$ is $C^{1+}$-smooth.
		Following the estimation in \cite[Lemma 7.1]{am2}, we can show that the holonomy map $h^s_Y$ converges to $h^s_X$ in $C^1$-topology as $Y\rightarrow X$ in $\mathscr{X}^2(M^3)$.
		
		Recall that the stable foliation $\cF^{ss}_Y$ is H\"older continuous. There exists a H\"older continuous function $T:J_2\rightarrow\RR$, such that for every $w_1\in J_1$ and $w_2=h^s_Y(w_1)\in J_2$, they satisfy 
		$$
		\phi_{T(w_2)}^X(w_2)~\in~\cF^{ss}_{Y,{\it loc}}(w_1).
		$$
		Here we can assume the function $|T|<1$ on $J_2$ by flowing the cross section $\Sigma$. 
		
		Since the strong stable foliation
		$\cF^{ss}_Y$ is preserved by $\phi_t^Y$, for $t$ small enough, we have 
		$$
		\phi_{T(w_2)+t}^X(w_2)~\in~
		\cF^{ss}_{Y,{\it loc}}(\phi_t^Y(w_1)).
		$$
		This implies there exists $\delta>0$, such that if 
		$$
		I_{\delta}~=~[x_1-\delta,x_1+\delta]\times\{y_1\}
		~\subset~(a_1,b_1)\times\{y_1\}~=~J_1,
		$$
		and $D^Y_{\delta}=\bigcup_{t\in[-\delta,\delta]}\phi_t^Y(I_\delta)$, then the holonomy map of the local strong stable foliation $\cF^{ss}_Y$ is well defined:
		$$
		h^{ss}_Y:~D^Y_{\delta}\longrightarrow 
		D^Y_2=\bigcup_{t\in[-1,1]}\phi_t^Y(J_2),
		$$
		and satisfies $h^{ss}_Y(w_1)=\phi_{T(w_2)}^X(w_2)$ where $w_2=h^s_Y(w_1)\in J_2$.
		
		The holonomy map $h^{ss}_Y:D^Y_{\delta}\rightarrow D^Y_2$ is absolutely continuous, see \cite{BP,pugh-shub}. Moreover, let $m_1$ and $m_2$ be Lebesgue measures induced by the Riemannian metric restricted on $D^Y_{\delta}$ and $D^Y_2$ respectively. Then for every $w_1\in D^Y_{\delta}$ and $w_2=h^{ss}_Y(w_1)\in D^Y_2$,
		the Jacobian of $h^{ss}_Y(w_1)$ is
		\begin{align}\label{equ:abso-continuity}
		\jac\left(h^{ss}_Y\right)(w_1)~=~
		\prod_{i=0}^{+\infty}
		\frac{\jac\left(D\phi_{-1}^Y|_{T_{\phi_i^Y(w_2)}\phi_i^Y(D^Y_2)}\right)}{\jac\left(D\phi_{-1}^Y|_{T_{\phi_i^Y(w_1)}\phi_i^Y(D^Y_{\delta})}\right)},
		\end{align}
		see \cite[Theorem 7.1 \& Remark 7.2]{Pe}. Notice here the infinite product must convergence since $w_2=h^{ss}_Y(w_1)\in\cF^{ss}_{Y,{\it loc}}(w_1)$.
		
		For every $\kappa>0$, there exists a $C^2$-neighborhood $\cV_{\kappa,1}\subset\mathscr{X}^2(M^3)$ of $X$, and an integer $k>0$, such that for every $Y\in\cV_{\kappa,1}$ and $w_2=h^{ss}_Y(w_1)$ with $w_1\in D^Y_{\delta}$, they satisfies
		\begin{align}\label{equ:tail}
		\sum_{i=k+1}^{+\infty}
		\left| \log\jac\left(D\phi_{-1}^Y|_{T_{\phi_i^Y(w_2)}\phi_i^Y(D^Y_2)}\right)-\log\jac\left(D\phi_{-1}^Y|_{T_{\phi_i^Y(w_1)}\phi_i^Y(D^Y_{\delta})}\right) \right| <\frac{\kappa}{4}.
		\end{align}
		Here we use the fact that the distance between $\phi_i^Y(w_1)$ and $\phi_i^Y(w_2)$, and the distance between the tangent bundles $T_{\phi_i^Y(w_1)}\phi_i^Y(D^Y_{\delta})$ and $T_{\phi_i^Y(w_2)}\phi_i^Y(D^Y_2)$ are both exponential convergent to zero.
		
		For every $w\in I_{\delta}$ and every $t\in[-\delta,\delta]$, we denote $w_t^X=\phi_t^X(w)\in D_{\delta}^X$ and $w_t^Y=\phi_t^Y(w)\in D_{\delta}^Y$ for $Y\in\cV_{\kappa,1}$.
		Let $w_2^X=h^{ss}_X(w_t^X)$ and $w_2^Y=h^{ss}_Y(w_t^Y)$.
		Equation \ref{equ:abso-continuity} and Equation \ref{equ:tail} show that
		\begin{align*}
		&\left| \log\jac\left(h^{ss}_Y\right)(w_t^Y)~-~
		\log\jac\left(h^{ss}_X\right)(w_t^X) \right| \\
		~\leq~&\sum_{i=0}^{k} \left|\log\jac\left(D\phi_{-1}^Y|_{T_{\phi_i^Y(w_2^Y)}\phi_i^Y(D^Y_2)}\right)- \log\jac\left(D\phi_{-1}^X|_{T_{\phi_i^X(w_2^X)}\phi_i^X(D^X_2)}\right)\right| \\
		&+\sum_{i=0}^{k} \left|\log\jac\left(D\phi_{-1}^Y|_{T_{\phi_i^Y(w_t^Y)}\phi_i^Y(D^Y_{\delta})}\right)- \log\jac\left(D\phi_{-1}^X|_{T_{\phi_i^X(w_t^X)}\phi_i^X(D^X_{\delta})}\right)\right|+\frac{\kappa}{2}.
		\end{align*}
		When $Y\rightarrow X$ in $C^2$-topology, we have $D\phi_{-1}^Y\rightarrow D\phi_{-1}^X$, and for every $i=0,\cdots,k$,
		\begin{itemize}
			\item $\phi_i^Y\big(D^Y_{\delta}\big)\rightarrow \phi_i^X\big(D^X_{\delta}\big)$ and $\phi_i^Y\big(D^Y_2\big)\rightarrow \phi_i^X\big(D^X_2\big)$ in $C^1$-topology;
			\item $\phi_i^Y\big(w_t^Y\big)\rightarrow \phi_i^X\big(w_t^X\big)$ and 
			$\phi_i^Y\big(w_2^Y\big)\rightarrow \phi_i^X\big(w_2^X\big)$.
		\end{itemize}
		So there exists a $C^2$-neighborhood $\cV_{\kappa,2}\subset\cV_{\kappa,1}$ of $X$, such that for every $Y\in\cV_{\kappa,2}$,
		$$
		\left| \log\jac\left(h^{ss}_Y\right)(w_t^Y)~-~
		\log\jac\left(h^{ss}_X\right)(w_t^X) \right|<\kappa.
		$$
		Thus we have $\jac\left(h^{ss}_Y\right)\rightarrow\jac\left(h^{ss}_X\right)$ as $Y\rightarrow X$ in $\mathscr{X}^2(M^3)$.
		
		\vskip1mm
		
		Let $\big\{(x,t): x\in[x_1-\delta,x_1+\delta],t\in[-\delta,\delta]\big\}$ be the coordinate on $D_{\delta}^Y$, which is defined as $w=(x,t)$ if $w_0=(x,y_1)\in I_{\delta}$ and $w=\phi_t^Y(w_0)$. Then there exists a smooth function $v^Y_1$ on $D_{\delta}^Y$, such that the Lebesgue measure $m_1$ on $D_{\delta}^Y$ induced by the Riemannian metric is
		$$
		m_1(x,t)~=~v^Y_1(x,t){\rm d}x{\rm d}t.
		$$
		Similarly, we have the coordinate $\big\{(x,t): x\in(a_2,b_2),t\in[-1,1]\big\}$ on $D_2$ and the smooth function $v^Y_2$ satisfying the Lebesgue measure $m_2(x,t)=v^Y_2(x,t){\rm d}x{\rm d}t$ on $D_2^Y$.
		
		According to the proof of \cite[Lemma 7.1]{am2}, for every $w_1=(x_1,y_1)\in J_1$ and $w_2^Y=(x_2^Y,y_2)=h^s_{Y,\Sigma}(w_1)\in J_2$, let $\phi_{T(w_2^Y)}^Y(w_2)\in\cF^{ss}_{Y,{\it loc}}(w_1)\cap D^Y_2$, then $w_1=(x_1,0)$ in the $(x,t)$-coordinate in $D^Y_{\delta}$, and $h^{ss}_Y(w_1)=\phi_{T(w_2^Y)}^Y(w_2)=(x_2,T(w_2^Y))$ in the $(x,t)$-coordinate in $D^Y_2$. The derivate of the holonomy map $h^s_{Y,\Sigma}$ is equal to
		\begin{align}\label{equ:holo-deriv}
		Dh^s_{Y,\Sigma}(w_1)~=~
		\lim_{\delta\rightarrow 0}
		\frac{\int_{D^Y_{\delta}}\jac\left(h^{ss}_Y\right)v_1(x,t) ~{\rm d}x{\rm d}t}{\int_{h^{ss}_Y\left(D^Y_{\delta}\right)}v_2(x,t) ~{\rm d}x{\rm d}t}
		~=~\jac\left(h^{ss}_Y\right)(w_1)\cdot
		\frac{v_1(x_1,0)}{v_2(x_2,T(w_2^Y))}.
		\end{align}
		When $Y\rightarrow X$ in $\mathscr{X}^2(M^3)$, $h^{ss}_Y(w_1)\rightarrow h^{ss}_X(w_1)$, so $x_2^Y\rightarrow x_2^X$ and $T(w_2^Y)\rightarrow T(w_2^X)$. Moreover, $D^Y_{\delta}\rightarrow D^X_{\delta}$ and $D^Y_2\rightarrow D^X_2$ imply $v^Y_i\rightarrow v^X_i$ for $i=1,2$. Thus Equation \ref{equ:holo-deriv} implies
		$$
		Dh^s_{Y,\Sigma}\longrightarrow Dh^s_{X,\Sigma}\equiv{\rm Id}_x
		\qquad {\rm as} \qquad
		Y\longrightarrow X~{\rm in}~\mathscr{X}^2(M^3).
		$$
		Thus for $\eta>0$ decided in (\ref{equ:eta}), there exists a $C^2$-neighborhood $\cU_2$ of $X$, such that for every $Y\in\cU_2$, it satisfies
		$$
		(1+\eta)^{-1}~<~
		Dh^s_{Y,\Sigma}
		~<~(1+\eta).
		$$ 
		This finishes the proof of Claim~\ref{clm:holonomy}.
	\end{proof}
	
	For every $Y\in\cU_1\cap\cU_2$, we give a new coordinate $\{(x',y')\}$ on the cross section $\Sigma$, such that the Poincar\'e map of $\phi_t^Y$ associated to $\Sigma$ satisfies the definition of Lorenz attractors in this coordinate. For every $z=(x_z,y_z)\in\Sigma$, we fix the $y'$-coordinate of $z$ to $y'_z=y_z$.
	
	For $\eta>0$ defined in (\ref{equ:eta}), there exists $\epsilon_{\eta}>0$, such that for every $x^-,x^+,x_0\in\RR$ satisfying
	$$
	|x^--(-1)|<\epsilon_{\eta}, \qquad
	|x^+-1|<\epsilon_{\eta}, \qquad {\rm and} \qquad
	|x_0|<\epsilon_{\eta},
	$$
	there exists a diffeomorphism $h:\RR\rightarrow\RR$ satisfying
	$$
	h(x^-)=-1, \quad
	h(x^+)=1, \quad 
	h(x_0)=0, \quad {\rm and} \quad
	(1+\eta)^{-1}<h'(x)<(1+\eta),
	\quad\forall x\in\RR.
	$$
	
	The unstable manifolds of $\sigma_Y$ intersects $\Sigma$ with $z^+_Y\rightarrow z^+=(-1,y^+)$ and $z^-_Y\rightarrow z^-=(1,y^+)$ as $Y\rightarrow X$. Moreover, the intersection between the local stable manifolds of $\sigma_Y$ and $\Sigma$ satisfies $W^s_{\it loc}(\sigma_Y,\phi_t^Y)\cap\Sigma=l_Y\rightarrow l=\{0\}\times[-1-\epsilon,1+\epsilon]$ as $Y\rightarrow X$. There exists a $C^2$-neighborhood $\cU_3$ of $X$, such that for every $Y\in\cU_3$, let $\cF^s_{Y,\Sigma}$ be the stable foliation on $\Sigma$ and $J_0=[-1-\epsilon,1+\epsilon]\times\{0\}$, then the coordinates
	$$
	\cF^s_{Y,\Sigma}(z^+_Y)\cap J_0=(x^-_Y,0), \quad
	\cF^s_{Y,\Sigma}(z^-_Y)\cap J_0=(x^+_Y,0), \quad {\rm and} \quad
	l_Y\cap J_0=(x_{0,Y},0)
	$$
	satisfy
	$$
	|x^-_Y-(-1)|<\epsilon_{\eta}, \qquad
	|x^+_Y-1|<\epsilon_{\eta}, \qquad {\rm and} \qquad
	|x_{0,Y}|<\epsilon_{\eta}.
	$$
	Let $h_Y:\RR\rightarrow\RR$ be the diffeomorphism satisfying
	$$
	h_Y(x^-_Y)=-1, \quad
	h_Y(x^+_Y)=1, \quad 
	h_Y(x_{0,Y})=0, \quad {\rm and} \quad
	(1+\eta)^{-1}<h_Y'(x)<(1+\eta),
	\quad\forall x\in\RR.
	$$
	Then we define $x'$-coordinate on $J_0$ as following: for every $z=(x_z,0)\in J_0$, the $(x',y')$-coordinate of $z$ is 
	$$
	(~x'_z,~y'_z~)~=~\big(~h_Y(x_z),~0~\big).
	$$
	
	Let $J_Y=\{(x,0):~-1\leq h_Y(x)\leq 1\}\subseteq J_0$, and 
	$$
	\Sigma_Y~=~\left(\bigcup_{z\in J_Y}\cF^s_{Y,\Sigma}\right)
	~\cap~
	\big\{(x,y)\in\Sigma:|y|\leq1\big\}.
	$$
	For every $z=(x_z,y_z)\in\Sigma_Y$, let $w=(x_w,0)=\cF^s_{Y,\Sigma}(z)\cap J_0$, then we define the $x'$-coordinate of $z$ as 
	$$
	(~x'_z,~y'_z~)~=~(~x'_w,~y_z~)~=~\big(~h_Y(x_w),~y_z~\big).
	$$
	Thus in the coordinate $\big\{(x',y')\big\}$, $\Sigma_Y$ is $C^1$-diffeomorphic to $[-1,1]^2$.
	
	For every $Y\in\cU_1\cap\cU_2\cap\cU_3$, 
    the stable foliation $\cF^s_{Y,\Sigma}$ is invariant under the action of Poincar\'e map, hence the Poincar\'e map $P_Y:\Sigma_Y\setminus l_Y\rightarrow\Sigma_Y$ is well defined and has the form
	$$
	P_Y(x',y')=\big(~f_Y(x'),~H_Y(x',y')~\big),
	\qquad\forall (x',y')\in\Sigma_Y\setminus l_Y.
	$$
	
	From the construction of $(x',y')$-coordinate on $\Sigma_Y$, the one dimensional quotient map $f_Y$ is $C^1$-smooth and satisfies
	$\lim_{x'\rightarrow0^-}f_Y(x')=1$, $\lim_{x'\rightarrow0^+}f_Y(x')=-1$, and $-1<f_Y(x')<1$ for every $x'\in[-1,1]\setminus\{0\}$. Since $Y\in\cU_1$, the Poincar\'e map $P_Y$ satisfies $\partial\pi_x\circ P_Y/\partial x>(1+\eta)^{-1}\lambda_0$ everywhere in the $(x,y)$-coordinate. So for every $z=(x_z,y_z)=(x'_z,y_z)\in\Sigma_Y\setminus l_Y$, it satisfies
	$$
	x'_z=h_Y\circ h^s_{Y,\Sigma}(x_z),
	\qquad i.e. \qquad
	x_z=\big( h^s_{Y,\Sigma} \big)^{-1}\circ h_Y^{-1}(x'_z).
	$$
	Let $w=P_Y(z)=(x_w,y_w)=(f_Y(x'_z),y_w)\in\Sigma_Y$, then 
	$$
	x_w=\pi_x\circ P_Y(z),
	\qquad {\rm and } \qquad
	f_Y(x'_z)=h_Y\circ h^s_{Y,\Sigma}(x_w)
	=h_Y\circ h^s_{Y,\Sigma}\circ\pi_x\circ P_Y(x_z,y_z).
	$$
	Since $y_z=y_z'$, we have
	\begin{align*}
	f_Y'(x'_z) ~&=~h'_Y \cdot Dh^s_{Y,\Sigma} \cdot 
	\frac{\partial \pi_x\circ P_Y}{\partial x} \cdot
	D\big( h^s_{Y,\Sigma} \big)^{-1} \cdot 
	\big(h_Y^{-1}\big)'(x'_z) \\
	~&\geq~ (1+\eta)^{-2}\cdot (1+\eta)^{-1}\lambda_0 \cdot (1+\eta)^{-2} \\
	~&\geq~(1+\eta)\sqrt{2}.
	\end{align*}
	This shows that one-dimensional quotient map $f_Y$ satisfies the property of geometric Lorenz attractors. 
	
	Since $P_Y$ is uniformly contracting in each leaf of $\cF^s_{Y,\Sigma}$, the function $H_Y$ is contracting along the $y'$-coordinate. The proof of rest properties for $H_Y$ is the same with $f_Y$. Thus we take $\cU=\cU_1\cap\cU_2\cap\cU_3$ which is a $C^2$-neighborhood of $X$, and $\Lambda_Y$ is a geometric Lorenz attractor for every $Y\in\cU$.    This completes the proof of Proposition~\ref{prop:Lorenz}.
\end{proof}

\section{Appendix: Singular hyperbolic attractors for $C^1$-vector fields}\label{Section:C1}

For vector fields on a closed smooth manifold $M$ of any dimension in $C^1$-topology, 
one already has the well known perturbation techniques such as Hayashi's connecting lemma \cite{hayashi,wen-xia-connecting,wen-uniform-connecting} and Bonatti-Crovisier's chain connecting lemma~\cite{bc}. 
Following similar arguments of Section~\ref{Section:measure} and applying the $C^1$-connecting lemmas, one obtains same conclusions as in Theorem~\ref{Thm:B} and Corollary~\ref{Cor:Cr-entropy-support} for  singular hyperbolic attractors of $X\in\mathcal{X}^1(M)$. We sketch the results for completeness.
We would like to point out  that researches on singular hyperbolic attractors about robust transitivity, hyperbolicity and  entropy theory can be found in a series works~\cite{ams,CY-Robust-attractor,mo-pa,sgw,PYY}. 
In particular, a more recent work  by S. Crovisier and D. Yang~\cite{CY-Robust-attractor} proved  that there exists an open dense set $\mathcal{U}$ in $\mathscr{X}^1(M)$ such that for any $X\in\mathcal{U}$, any non-trivial singular hyperbolic attractor $\Lambda$ is a robustly transitive attractor and $\Lambda$ is a homoclinic class in which all periodic orbits are homoclinically related to each other.
This ensures that one could follow similar arguments of Section~\ref{Section:measure} and applying the $C^1$-connecting lemmas to study the space of ergodic measures for singular hyperbolic attractors of $X\in\mathcal{X}^1(M)$.
Similarly as in Theorem~\ref{Thm:B}, the space of ergodic measures restricted on  singular hyperbolic attractors is path connected for $C^1$-generic vector fields while on the other hand, for $C^1$-dense vector fields admitting a  singular hyperbolic attractor with a co-dimensional two singular hyperbolic splitting, the singular measure is isolated in the space of ergodic measures, which implies in this case that the ergodic measure space is not connected.

\begin{theorem}\label{Thm:SH-attractor}
	There exist a residual subset $\mathscr{R}\subset\mathscr{X}^1(M)$ and a dense subset $\mathscr{D}\subset\mathscr{X}^1(M)$, such that
	\begin{itemize}
		\item if $X\in\mathscr{R}$ and $\Lambda$ is a singular hyperbolic attractor of $X$, then $\mathcal{M}_{inv}(\Lambda)=\overline{\mathcal{M}_{erg}(\Lambda)}=\overline{\mathcal{M}_{per}(\Lambda)}$. Moreover, the space $\mathcal{M}_{erg}(\Lambda)$ is path connected.
		\item if $X\in\mathscr{D}$ and $\Lambda$ is a singular hyperbolic attractor of $X$ whose singular hyperbolic splitting $T_{\Lambda}M=E^{ss}\oplus E^{cu}$ satisfies $\dim(E^{cu})=2$, then $\overline{\mathcal{M}_{per}(\Lambda)}\subsetneqq\overline{\mathcal{M}_{erg}(\Lambda)}\subsetneqq\mathcal{M}_{inv}(\Lambda)$. Moreover, the space $\mathcal{M}_{erg}(\Lambda)$ is not connected.	
	\end{itemize}
\end{theorem}

Concerning  the residual part of Theorem~\ref{Thm:SH-attractor}, similar properties are obtained in~\cite[Theorem 3.5]{abc} for isolated non-trivial transitive sets of diffeomorphisms. For vector fields, one has to conquer the difficulties caused by the existence of singularities, i.e. to show that the atomic measure of the singularity can be accumulated by periodic measures inside the singular hyperbolic attractor. This could be achieved by applying Hayashi's connecting lemma combined with  the arguments in Section~\ref{Section:ergodic-measure-Lorenz}. An alternative proof can be found in~\cite[Theorem 3.1]{Yang-Zhang2}  where they proved $\mathcal{M}_{inv}(\Lambda)=\overline{\mathcal{M}_{erg}(\Lambda)}=\overline{\mathcal{M}_{per}(\Lambda)}$ for which $\Lambda$ is a non-trivial isolated transitive set for $C^1$-generic vector fields.

Similarly as Theorem~\ref{Thm:B}, the dense part of Theorem~\ref{Thm:SH-attractor}  is a sharp comparison with the residual part. For the dense part, one applies Hayashi's connecting lemma to make the two branches of the  unstable manifolds of the singularity with co-index one lie in the stable manifold of a periodic orbit. Then the arguments follows the dense part in Section~\ref{Section:ergodic-measure-Lorenz}, see also Remark~\ref{Rem:isolated-measure}.
A related work by C. Morales~\cite{Morales} proved that  every non-transitive sectional-Anosov flow with dense periodic orbits exhibits an invariant measure that can not be approximated by ergodic ones. The arguments in~\cite{Morales} is quite different with us where the author considered invariant measures that are not supported on a transitive set, while we obtain here isolated ergodic measures inside an attractor.

The following is a corollary of Theorem~\ref{Thm:SH-attractor} which concerns the support and entropy of measures on singular hyperbolic attractors and whose statement is similar to Corollary~\ref{Cor:Cr-entropy-support}.
\begin{corollary}\label{Cor:entropy-support}
	There exists a residual subset $\mathscr{R}\subset\mathscr{X}^1(M)$, such that for any  $X\in\mathscr{R}$, if $\Lambda$ is a singular hyperbolic attractor of $X$, then:
	\begin{enumerate}
		\item\label{item:residual} there exists a residual subset $\mathcal{M}_{res}$ in $\mathcal{M}_{inv}(\Lambda)$ such that every $\mu\in\mathcal{M}_{res}$ is ergodic with $\supp(\mu)=\Lambda$ and $h_{\mu}=0$.
		
		\item\label{item:dense} there exits a dense subset $\mathcal{M}_{den}$ in $\mathcal{M}_{inv}(\Lambda)$ such that every $\mu\in\mathcal{M}_{den}$ is ergodic with $h_{\mu}>0$.
	\end{enumerate}
\end{corollary}

\begin{remark}
	Recently, C. Bonatti and A. da Luz~\cite{bo-daluz} introduced  another notion called {\it multisingular hyperbolicity} to study more general vector fields with singularities. An interesting question is whether conclusions of Theorem~\ref{Thm:SH-attractor} is true or not for miultisingular hyperbolic chain recurrence classes.	
\end{remark}

\bibliographystyle{plain}

\flushleft{\bf Yi Shi} \\
School of Mathematics, Sichuan University, Chengdu, 610065, China\\
\textit{E-mail:} \texttt{shiyi@scu.edu.cn}\\

\flushleft{\bf Xueting Tian} \\
School of Mathematical Sciences, Fudan University, Shanghai, 200433,  P.R. China\\
\textit{E-mail:} \texttt{xuetingtian@fudan.edu.cn}\\

\flushleft{\bf Xiaodong Wang} \\
School of Mathematical Sciences, CMA-Shanghai, Shanghai Jiao Tong University, Shanghai, 200240, P.R. China\\
\textit{E-mail:} \texttt{xdwang1987@sjtu.edu.cn}\\

\end{document}